\documentclass[]{article}
\usepackage{amsmath}
\usepackage{amssymb}
\usepackage{hyperref}
\usepackage{pgf,pgfarrows,pgfnodes,pgfautomata,pgfheaps,pgfshade}
\usepackage{graphicx,amsfonts}
\usepackage{amsmath,amssymb,setspace}
\usepackage{mathrsfs}
\usepackage{colortbl}
\usepackage{tikz}
\usepackage{framed}
\usepackage{subfig}

\numberwithin{equation}{section}

\newtheorem{thm}{Theorem}[section]

\newtheorem{theorem}[thm]{Theorem}

\newtheorem{proposition}[thm]{Proposition}
\newtheorem{definition}[thm]{Definition}

\newenvironment{proof}{{\bf Proof.}}{\hfill$\square$\vskip.5cm}

\title{Some Convergence Results on the Regularized Alternating Least-Squares Method for Tensor Decomposition}

\author{Na Li\footnote{Department of Mathematics, Clarkson University, Potsdam, NY 13699, USA, [nali,cnavasca]@clarkson.edu.} \,\, Stefan Kindermann\footnote{Industrial Mathematics Institute, Johannes Kepler Universitat Linz, Altenbergerstrasse 69, A-4040 Linz, Austria , kindermann@indmath.uni-linz.ac.at} \,\,  Carmeliza Navasca${}^*$\footnote{Corresponding author.}}

\date{\today}

\begin{document}

\maketitle

\begin{abstract}
\setcounter{section}{0}

We study the convergence of the Regularized Alternating Least-Squares algorithm for tensor decompositions. 
As a main result, we have shown that given the existence of critical points of the Alternating Least-Squares method, the limit points of  the converging subsequences of the RALS are the critical points of the least squares cost functional. Some numerical examples indicate a faster convergence rate for the RALS in comparison to the usual alternating least squares method. \\ 

\vspace{8pt}
\noindent
\end{abstract}

\section{Introduction}

A well-known iterative method for CANDECOM/PARAFAC (CP) is the Alternating Least-Squares (ALS) technique. Independently, the ALS was introduced by Carol and Chang \cite{JJ} and Harshman \cite{RAH} in 1970. It has been extensively applied to many problems across various engineering \cite{Sid1}\cite{Sid2}\cite{Acar} \cite{DeVos} and science \cite{Smilde}\cite{Kroonenberg} fields; see the survey papers \cite{KB} \cite{Lieven} and the references therein.  For example, Beylkin and Mohlenkamp \cite{Beylkin} \cite{BM2} utilizes ALS to compute optimal separation rank for certain operators like inverse Laplacian and the multiparticle Schr\"odinger equation to reduce computational complexity. In a more recent application, Doostan et al. \cite{Doostan} has implemented ALS to study complex systems modeled by stochastic PDEs.

Its widespread success can be attributed to the simplicity of the method.  Moreover, Bro et al. \cite{Bro1} \cite{GR} found that the ALS method gives superior quality solutions with fewer memory and time requirements than the other CP methods. Despite its success, the ALS has some drawbacks.  For example, initialization of the factor matrices, collinearity in the factor matrices or degeneracy problems may require a high number of iterations for the ALS method to converge.  This slowed convergence characterized by a flat curve in a log error plot is referred to as the \emph{swamp}. Swamps can be present in the non-degenerate and degenerate cases. The degenerate case is a more challenging problem; see \cite{PP} \cite{LimComon} for some regularization techniques for the degenerate swamps.
 
Here we address the non-degenerate case. There have been several methods which address the issues of the swamp occurrences in the non-degenerate case. For example, line search schemes \cite{Comon}\cite{Nion} have been used to accelerate the ALS algorithm.  An entirely different approach by De Lathauwer, De Moor and Vandewalle obviates the swamp issues by considering a simultaneous matrix diagonalization for CP decomposition \cite{JustLieven} \cite{dLdMvdW}. Paatero \cite{PP2} have applied regularization to a gradient descent based method for CP.

In this paper, we analyze the Regularized Alternating Least-Squares (RALS) method introduced by Navasca, Kindermann and De Lathawer \cite{NDLK}.  The implementation of RALS is simple; it is no more complicated than the ALS algorithm. The cost functional of the RALS penalizes the difference between the current and previous factor iterates with a regularization parameter. Unlike the tensor regularization method found in \cite{PP} \cite{LimComon}, RALS is an unconstrained optimization problem since there is no uniform constraint in the penalty terms that are sequentially changing at each iteration, and thus, the sequences of limit points of RALS can be unbounded.
Hence, RALS does not address the degeneracy problem; i.e. RALS will not find a critical point if the original ALS functional does not have a critical point. 

The study of the convergence analysis of ALS and RALS is facilitated by an optimization framework; ALS is the nonlinear block Gauss-Seidel (GS) and RALS is the nonlinear block  proximal point modification of GS (PGS) for CP tensor decomposition. What we have shown is that if a limit point exists, then it is a critical point of the functional. More specifically, we study the ALS non-degenerate swamps by analyzing how RALS removes, if not, shortens the swamps. Furthermore, our convergence analysis brings attention to fact that the RALS functional has a weakened assumption than the requirement of the ALS functional. This finding sheds light on the swamps in the non-degenerate case in the ALS method.

\subsection{Organization}

Beginning with Section \ref{sec: two}, we give some preliminaries which include basic definitions of rank-one tensor and CP decomposition. Section \ref{alswithgs} reviews the classical ALS for third-order tensors and includes a discussion on the ALS swamp through an example. Section \ref{RALS} is the main section where we introduce the RALS method and show that if the sequence obtained from RALS algorithm converges, then the limit points are the critical points of the original ALS algorithm.  In Section $5$, we provide a numerical comparison study of ALS and RALS with several data sets. Lastly, we make some concluding remarks in Section $6$.

\section{Preliminaries}\label{sec: two}

We denote the scalars in $\mathbb{R}$ with lower-case letters $(a,b,\ldots)$ and the vectors with bold lower-case letters $(\bf{a},\bf{b},\ldots)$.  The matrices are written as bold upper-case letters $(\bf{A}, \bf{B},\ldots)$ and the symbol for tensors are calligraphic letters $(\mathcal{A},\mathcal{B},\ldots)$. The subscripts represent the following scalars:  $\mathcal{(A)}_{ijk}=a_{ijk}$, $(\bold{A})_{ij}=a_{ij}$, $(\bold{a})_i=a_i$ and the $r$-th column of a matrix $\bold{A}$ is $\bold{a_r}$. The matrix sequence is $\{\bold{A}^k\}$.

\begin{definition}
The Kronecker product of matrices $\bold{A}$ and $\bold{B}$ is defined as
{\small \begin{eqnarray*}
\bold{A} \otimes \bold{B}= \left[
\begin{array}{ccc}
a_{11}\bold{B}  & a_{12}\bold{B} & \ldots \\
a_{21}\bold{B}  & a_{22}\bold{B} & \ldots \\
\vdots                  & \vdots         
\end{array} \right].
\end{eqnarray*}}
\end{definition}

\begin{definition}
The Khatri-Rao product is the ``matching columnwise" Kronecker product. Given matrices $\bold{A} \in \mathbb{R}^{\mathnormal{I} \times \mathnormal{K}}$ and $\bold{B} \in \mathbb{R}^{\mathnormal{J} \times \mathnormal{K}}$, their Khatri-Rao product is denoted by $\bold{A} \odot \bold{B}$. The result is a matrix of size $(\mathnormal{IJ} \times \mathnormal{K})$ defined by
\begin{eqnarray*}
\bold{A} \odot \bold{B}= [\bold{A_1} \otimes \bold{B_1}~~ \bold{A_2} \otimes \bold{B_2}~~ \ldots].
\end{eqnarray*}
\end{definition}


\begin{definition}[Mode-$n$ fibers]
A mode-$n$ fiber of an $N$th order tensor is a vector defined by fixing all indices but the $n$-th one.
\end{definition}

For example, a matrix column is a mode-$1$ fiber and a matrix row is a mode-$2$ fiber. Third-order tensors have column (mode-$1$), row (mode-$2$) and tube (mode-$3$) fibers, denoted by $\bold{x_{:jk}}$, $\bold{x_{i:k}}$ and $\bold{x_{ij:}}$ respectively. 

\begin{definition}[Mode-$n$ matricization]
Matricization is the process of reordering the elements of an $N$th order tensor into a matrix. The mode-$n$ matricization of a tensor $\mathcal{T} \in \mathbb{R}^{I_1 \times I_2 \times \cdots \times I_N}$ is denoted by $\bold{T_{(n)}}$ and concatenates the mode-$n$ fibers to be the columns of the resulting matrix.
\end{definition}

If we use a map to express such matricization process for any $N$th order tensor $\mathcal{T} \in \mathbb{R}^{I_1 \times I_2 \times \cdots \times I_N}$, that is, the tensor element $(i_1, i_2, \dots, i_N)$ maps to matrix element $(i_n, j)$, then there is a formula to calculate $j$: 
$$j=1+\sum_{\substack{k=1\\ k\neq n}}^{N}(i_k -1)J_k \quad \text{with} \quad J_k =\prod_{\substack{m=1\\ m\neq n}}^{k-1}I_m.$$
So, given a third-order tensor $\mathcal{X} \in \mathbb{R}^{I \times J \times K}$, the mode-$1$, mode-$2$ and mode-$3$ matricizations of $\mathcal{X}$ are:
\begin{eqnarray*}
 \bold{X_{(1)}}&=&[\bold{x_{:11}}, \dots, \bold{x_{:J1}}, \bold{x_{:12}}\dots,\bold{x_{:J2}}, \dots, \bold{x_{:1K}}, \dots, \bold{x_{:JK}}] ,\\
  \bold{X_{(2)}}&=&[\bold{x_{1:1}}, \dots, \bold{x_{I:1}}, \bold{x_{1:2}}\dots,\bold{x_{I:2}}, \dots, \bold{x_{1:K}}, \dots, \bold{x_{I:K}}] ,\\
   \bold{X_{(3)}}&=&[\bold{x_{11:}}, \dots, \bold{x_{I1:}}, \bold{x_{12:}}\dots,\bold{x_{I2:}}, \dots, \bold{x_{1J:}}, \dots, \bold{x_{IJ:}}].
  \end{eqnarray*}

\begin{definition}[Rank-one tensor]
An $N$th order tensor $\mathcal{T} \in \mathbb{R}^{\mathnormal{I}_1 \times \mathnormal{I}_2 \times \cdots \times \mathnormal{I}_N}$ is a rank-one tensor if it can be written as the outer product of $N$ vectors, i.e.,
$$\mathcal{T} = \bold{a}^{(1)} \circ  \bold{a}^{(2)} \circ \cdots \circ  \bold{a}^{(N)},$$
where $\bold{a}^{(r)} \in \mathbb{R}^{I_r \times 1}, 1 \leq r \leq N$. The symbol `` $\circ$" represents the vector outer product. This means that each element of the tensor is the product of the corresponding vector elements: 
$$t_{i_1 i_2 \cdots i_N} = a_{i_1}^{(1)}a_{i_2}^{(2)} \cdots a_{i_N}^{(N)}, \quad \text{for all} \; 1 \leq i_n \leq I_n.$$
\end{definition}

\section{ALS and Nonlinear Block Gauss-Seidel Method}\label{alswithgs}

In 1927, Hitchcock \cite{Hitch1} \cite{Hitch2} proposed the idea of the polyadic form of a tensor, i.e., expressing a tensor as the sum of a finite number of rank-one tensors. Currently, this decomposition is called the CANDECOMP/PARAFAC (CP) decomposition. The Parallel Factor Decomposition (PARAFAC) first appeared in \cite{RAH} in the context of psychometrics. Independently, \cite{JJ} introduced this decomposition as the Canonical Decomposition (CANDECOMP) in phonetics. The work of Kruskal in 1977 \cite{JB} \cite{JBK} provided
 a sufficient condition,
\[ I + J + K \geq 2R + 2,   \]
where  $I$, $J$ and $K$ denote the {\em k-rank} (defined as the maximum value $k$ such that any $k$ columns are linearly independent in a matrix) of matrices $\bold{A}$, $\bold{B}$ and $\bold{C}$ respectively, for uniqueness up to permutation and scalings of CP. Later on, De Lathauwer \cite{JustLieven} and Jiang and Sidiropoulous \cite{Jiang} gave new sufficient conditions for uniqueness by assuming only one full-rank factor with the new bound,
\[\frac{R(R-1)}{2} \leq \frac{I(I-1)J(J-1)}{4}.     \]
Also, some constraints on the factor matrices of the CP are considered by requiring all the columns in each factor matrix to be orthonormal. This condition is useful in applications like independent component analysis (ICA) \cite{JF}. 

In terms of numerical methods for computing CP decomposition, there are several methods (see e.g. \cite{GR}) to solve CP decomposition of a given tensor. The ALS method is the most popular technique. We will discuss the connection of ALS to the nonlinear block Gauss-Seidel (GS) method  \cite{DPB} \cite{LM}. This connection is important since it is a well-known fact that the GS method does not necessarily converge, leading us to further discuss some convergence results of the ALS algorithm. 

\subsection{Alternating Least Squares}
For the simplicity of the exposition, we looked at  third-order tensors, but all the analysis holds for higher-order tensors. For a given third-order tensor $\mathcal{T} \in \mathbb{R}^{\mathnormal{I} \times \mathnormal{J} \times \mathnormal{K}}$, its CP decomposition is
\begin{eqnarray} \label{decomp1}
\mathcal{T}\approx \sum_{r=1}^{R} \bold{a_r} \circ \bold{b_r} \circ \bold{c_r}.
\end{eqnarray}

The {\em factor matrices} are the combination of the vectors from the rank-one components; i.e., $\bold{A}=[\bold{a_1}, \bold{a_2}, \cdots, \bold{a_{R}}] \in \mathbb{R}^{I\times R}$, $\bold{B}=[\bold{b_1}, \bold{b_2}, \cdots, \bold{b_R}] \in \mathbb{R}^{J \times R}$ and $\bold{C}=[\bold{c_1}, \bold{c_2}, \cdots, \bold{c_R}] \in \mathbb{R}^{K\times R}$ where $R$ is called the {\em rank} of the tensor $\mathcal{T}$ denoted by $R=\mbox{rank}(\mathcal{T})$.

The problem we want to solve is the following: given a third-order tensor $\mathcal{T} \in \mathbb{R}^{I\times J\times K}$, compute its CP decomposition with $R$ components of rank-one tensors that best approximates $\mathcal{T}$, i.e.,
\begin{eqnarray} \label{alsprob}
\displaystyle\mathop{\mathrm{minimize}}_{\widehat{\mathcal{T}}} \quad \Vert\mathcal{T}-\widehat{\mathcal{T}}\Vert_F^2,\quad \text{where} \quad \widehat{\mathcal{T}}=\sum_{r=1}^{R}\bold{a_r} \circ \bold{b_r} \circ \bold{c_r},
\end{eqnarray}
and $\Vert \cdot \Vert_F^2$ is the Frobenius norm. The problem is equivalent to
\begin{eqnarray}\label{alsprob2}
\displaystyle\mathop{\mathrm{min}}_{\bold{A},\bold{B},\bold{C}} \quad \Vert\mathcal{T}-\sum_{r=1}^{R}\bold{a_r} \circ \bold{b_r} \circ \bold{c_r}\Vert_F^2
\end{eqnarray}
with respect to factor matrices $\bold{A}$, $\bold{B}$ and $\bold{C}$.

By using the Khatri-Rao product and tensor matricization, (\ref{decomp1}) can be written in three matricized forms:
\begin{eqnarray*}
&&\bold{T_{(1)}}\approx \bold{A}(\bold{C}\odot\bold{B})^{\text{T}}, \\
&&\bold{T_{(2)}}\approx \bold{B}(\bold{C}\odot\bold{A})^{\text{T}}, \\
&&\bold{T_{(3)}}\approx \bold{C}(\bold{B}\odot\bold{A})^{\text{T}} .
\end{eqnarray*}

Then by fixing all factor matrices but one, the problem reduces to three coupled linear least-squares subproblems. Thus, ALS solves three least-squares subproblems to obtain the factor matrices through these subproblems: 

\begin{eqnarray}\label{als}
\bold{A}^{k+1}&=&\displaystyle\mathop{\mathrm{argmin}}_{\widehat{\bold{A}}\in \mathbb{R}^{I\times R}}\Vert\bold{T_{(1)}}^{I\times JK}-\widehat{\bold{A}}  (\bold{C}^{k}\odot\bold{B}^{k})^{\text{T}}\Vert_F^2, \notag \\
\bold{B}^{k+1}&=&\displaystyle\mathop{\mathrm{argmin}}_{\widehat{\bold{B}}\in \mathbb{R}^{J\times R}}\Vert\bold{T_{(2)}}^{J\times IK}-\widehat{\bold{B}}(\bold{C}^{k}\odot\bold{A}^{k+1})^{\text{T}}\Vert_F^2, \\
\bold{C}^{k+1}&=&\displaystyle\mathop{\mathrm{argmin}}_{\widehat{\bold{C}}\in \mathbb{R}^{K\times R}}\Vert\bold{T_{(3)}}^{K\times IJ}-\widehat{\bold{C}}(\bold{B}^{k+1}\odot\bold{A}^{k+1})^{\text{T}}\Vert_F^2,  \notag
\end{eqnarray}
where $\bold{T_{(1)}}^{I\times JK}$, $\bold{T_{(2)}}^{J\times IK}$ and $\bold{T_{(3)}}^{K\times IJ}$ are the mode-1, mode-2 and mode-3 matricizations of tensor $\mathcal{T}$. So starting from an initial guess $\bold{A}^0$, $\bold{B}^0$, $\bold{C}^0$, the ALS approach fixes $\bold{B}$ and $\bold{C}$ to solve for $\bold{A}$, then fixes $\bold{A}$ and $\bold{C}$ to solve for $\bold{B}$, and then fixes $\bold{A}$ and $\bold{B}$ to solve for $\bold{C}$. This process continues iteratively until some convergence criterion is satisfied. Therefore, this method translates the original nonlinear minimization problem to three subproblems where each one is just a least-squares problem. 

\subsection{Block Nonlinear Gauss-Seidel Method}\label{bngsm}

In this section, we want to introduce the nonlinear block Gauss-Seidel method  \cite{AAA} \cite{DPB} \cite{DPP} \cite{LM} \cite{LGM}, a technique that is used to find a minimizer of a nonlinear functional. We will see that ALS is a special case of GS. 

Consider such a problem:
\begin{eqnarray}\label{GS}
&&\text{minimize} \quad f(\bold{x})  \\
&&\text{subject to} \quad \bold{x}\in X = X_1 \times X_2 \times \cdots X_m \subseteq \mathbb{R}^n, \notag
\end{eqnarray}
where $f$ be a continuously differentiable function from $\mathbb{R}^n$ to $\mathbb{R}$ and $X$ is the cartesian product of closed, nonempty and convex subsets $X_i \subseteq \mathbb{R}^{n_i}$, for $i=1,\dots, m$, with $\displaystyle\sum_{i=1}^{m}n_i = n$. If the vector $\bold{x} \in \mathbb{R}^n$ is partitioned into $m$ component vectors $\bold{x_i} \in \mathbb{R}^{n_i}$, then we can consider $f$ is a function from $\mathbb{R}^{n_1} \times \mathbb{R}^{n_2} \times \cdots \mathbb{R}^{n_m}$ to $\mathbb{R}$ with
$$f(\bold{x})=f(\bold{x_1}, \bold{x_2}, \cdots, \bold{x_m}).$$
The minimization of the block nonlinear Gauss-Seidel method for the solution (\ref{GS}) is defined by the iteration,
$$\bold{x_i}^{k+1}=\displaystyle\mathop{\mathrm{argmin}}_{\bold{y_i} \in X_i} f(\bold{x_1}^{k+1}, \dots, \bold{x_{i-1}}^{k+1}, \bold{y_i}, \bold{x_{i+1}}^{k}, \dots, \bold{x_m}^{k}),$$
which in turn updates the components of $\bold{x}$, starting from a given initial guess $\bold{x}^0 \in X$ and generating a sequence $\{\bold{x}^k\}=\{(\bold{x_1}^k, \bold{x_2}^k, \dots, \bold{x_m}^k)\}$.

We introduce the following definitions to facilitate our discussion on the connection between GS and ALS.
\begin{definition}[Vectorization]
The vectorization of a matrix 
$$\bold{M} =[\bold{m_1}, \bold{m_2},\cdots, \bold{m_n}] \in \mathbb{R}^{m \times n},$$
 where $\bold{m_i}$ is the $i$-th column of $\bold{M}$, is denoted by $vec(\bold{M})$ which is a vector of size $(mn)$ defined by
$$vec(\bold{M})=\begin{bmatrix}
\bold{m_1} \\
\bold{m_2} \\
\vdots \\
\bold{m_n}
\end{bmatrix}.$$
\end{definition}

From the PARAFAC formulation (\ref{alsprob}) and the definition of rank-one tensor, the cost function we want to minimize is
$$\Vert\mathcal{T}-\widehat{\mathcal{T}}\Vert_F^2=\displaystyle\sum_{k=1}^{K}\sum_{j=1}^{J}\sum_{i=1}^{I}(t_{ijk}-\sum_{r=1}^{R}a_{ir}b_{jr}c_{kr})^2=f(\bold{A},\bold{B},\bold{C}),$$
where the cost function is a function s.t. $f: \mathbb{R}^n \rightarrow \mathbb{R}$, $n=(I+J+K)R$. Let $\bold{x}=vec([vec(\bold{A}), vec(\bold{B}), vec(\bold{C})]) \in \mathbb{R}^n$, it is obvious that 
$$f(\bold{x})=f(\bold{A},\bold{B},\bold{C})=\sum_{k=1}^{K}\sum_{j=1}^{J}\sum_{i=1}^{I}(t_{ijk}-\sum_{r=1}^{R}a_{ir}b_{jr}c_{kr})^2.$$ 
Let $vec(\bold{A})=\bold{x_1}\in \mathbb{R}^{IR}$, $vec(\bold{B})=\bold{x_2}\in \mathbb{R}^{JR}$ and $vec(\bold{C})=\bold{x_3}\in \mathbb{R}^{KR}$ so that we partition the vector $\bold{x} \in \mathbb{R}^n$ into 3 component vectors $\bold{x_i}\in \mathbb{R}^{n_i}$, $i=1,2,3$. $n_1=IR$, $n_2=JR$ and $n_3=KR$. It follows that $\bold{x}=\bold{x_1} \times \bold{x_2} \times \bold{x_3} \in \mathbb{R}^{n_1} \times \mathbb{R}^{n_2} \times \mathbb{R}^{n_3}=\mathbb{R}^n$. Thus, the CP decomposition can be reformulated to the following problem:
\begin{eqnarray}\label{GSA}
&&\text{minimize} \quad f(\bold{x})  \\
&&\text{subject to} \quad \bold{x}\in \mathbb{R}^{n_1} \times \mathbb{R}^{n_2} \times \mathbb{R}^{n_3}=\mathbb{R}^n.\notag
\end{eqnarray}
From the ALS algorithm, the updates are in terms of components of $\bold{x}$, starting from a given initial point $\bold{x}^0 = vec ([vec(\bold{A}^0), vec(\bold{B}^0), vec(\bold{C}^0)]) \in \mathbb{R}^{n}$ and generates a sequence $\{(\bold{x_1}^k, \bold{x_2}^k, \bold{x_3}^k)\}$ by the following: 
\begin{eqnarray*}
&\bold{x_{1}}^{k+1}&=\displaystyle\mathop{\mathrm{argmin}}_{\bold{y_1} \in \mathbb{R}^{n_1}}f(\bold{y_1}, \bold{x_{2}}^{k}, \bold{x_{3}}^{k}),\\
&\bold{x_{2}}^{k+1}&=\displaystyle\mathop{\mathrm{argmin}}_{\bold{y_2} \in \mathbb{R}^{n_2}}f(\bold{x_{1}}^{k+1}, \bold{y_2}, \bold{x_{3}}^{k}),\\
&\bold{x_{3}}^{k+1}&=\displaystyle\mathop{\mathrm{argmin}}_{\bold{y_3} \in \mathbb{R}^{n_3}}f(\bold{x_{1}}^{k+1}, \bold{x_{2}}^{k+1}, \bold{y_3}).
\end{eqnarray*}
Notice that this is the exact GS method to solve the problem (\ref{GSA}). Therefore, the ALS algorithm is the block nonlinear Gauss-Seidel method for solving the CP decomposition of a given tensor. 

\subsection{Some Analysis about ALS}\label{analysis-als}

Since we already know that the ALS method coincides with the GS method, we can bring some GS results to analyze the ALS algorithm. 
\begin{definition}[Critical Point]
Let $f:  X \rightarrow \mathbb{R}$, $X \subset \mathbb{R}^n$ is a continuously differentiable function, a {\bf critical point} of $f$ is a point $\overline{\bold{x}} \in X$ such that 
\begin{eqnarray}\label{condition}
\triangledown f(\overline{\bold{x}})^{\text{T}}(\bold{y}-\overline{\bold{x}}) \geq 0, \quad \forall \; \bold{y} \in X,
\end{eqnarray}
where $\triangledown f(\bold{x}) \in \mathbb{R}^{n}$ denotes the gradient of $f$ at $x$ and $\triangledown f(\overline{\bold{x}})^{\text{T}}$ is the transposition of it. If $X = \mathbb{R}^n$ or if  $\overline{\bold{x}}$ is an iterior point of $X$, then the condition (\ref{condition}) reduces to the stationarity condition $\triangledown f(\overline{\bold{x}})=\bold{0}$ of unconstrained optimization.
\end{definition}

\begin{theorem}[Optimality Condition]\label{opt}
(a) If $\overline{\bold{x}}$ is a local minimum of $f$ over $X$, then it satisfies the optimality condition (\ref{condition}), i.e.,
$$\triangledown f(\overline{\bold{x}})^{\text{T}}(\bold{x}-\overline{\bold{x}}) \geq 0, \quad \forall \; \bold{x} \in X.$$

(b) If $f$ is convex over $X$, then the condition of part (a) is also sufficient for $\overline{\bold{x}}$ to minimize $f$ over $X$.

If $X=\mathbb{R}^n$ or if $\overline{\bold{x}}$ is an interior point of $X$, then the condition reduces to $\triangledown f(\overline{\bold{x}})=\bold{0}$.
\end{theorem}
\begin{definition}[Limit Point]
We say that a vector $\bold{x} \in \mathbb{R}^n$ is a {\bf limit point} of a sequence $\{\bold{x}^k\}_{k=1}^{\infty}$ in $\mathbb{R}^n$ if there exists a subsequence of $\{\bold{x}^k\}_{k=1}^{\infty}$ that converges to $\bold{x}$.
\end{definition}

\begin{definition}[Convex Function]
A real-valued function $f(x)$ defined on a convex subset is called {\bf convex} if for any two points $x_1$ and $x_2$, in its domain and any $t \in [0, 1]$, 
$$f(tx_1 + (1-t)x_2) \leq tf(x_1)+(1-t)f(x_2).$$
If furthermore, 
$$f(tx_1 + (1-t)x_2) < tf(x_1)+(1-t)f(x_2),$$
$x_1 \neq x_2$, then $f$ is {\bf strictly convex}.
\end{definition}

\begin{definition}[Quasiconvex Function]\label{quasi}
A function $f: S \rightarrow \mathbb{R}$ defined on a convex subset S of a real vector space is {\bf quasiconvex} if whenever $x,y \in S$ and $\lambda \in [0,1]$, then
$$f(\lambda x+(1-\lambda)y) \leq max (f(x), f(y)).$$
If furthermore,
$$f(\lambda x+(1-\lambda)y) < max (f(x), f(y)),$$
 $x \neq y$, then $f$ is {\bf strictly quasiconvex}.

Consider the function $f$ in (\ref{GS}), which is defined on a subset $X=X_1 \times X_2 \times \cdots \times X_m$, we say that $f$ is quasiconvex with respect to $x_i \in X_i$ on $X$ if for every $x \in X$ and $y_i \in X_i$, we have 
$$f(x_1, \dots, tx_i + (1-t)y_i, \dots, x_m) \leq max\{f(x), f(x_1, \dots, y_i, \dots, x_m)\},$$
for all $t\in (0,1)$. If furthermore, 
$$f(x_1, \dots, tx_i + (1-t)y_i, \dots, x_m) < max\{f(x), f(x_1, \dots, y_i, \dots, x_m)\},$$
$y_i \ne x_i$, then $f$ is strictly quasiconvex.

\end{definition}

\begin{figure}[h]
\centering
\subfloat[Quasiconvex function, but not convex]{\label{fig: quasi}\includegraphics[height=4cm]{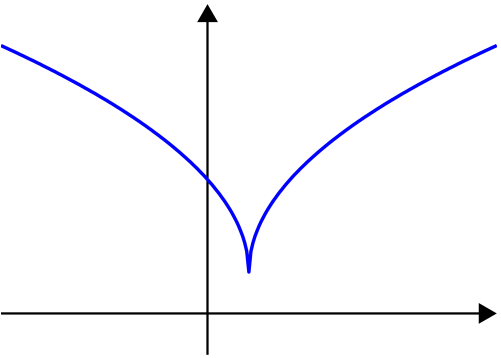}}
\qquad\qquad
\subfloat[Not a quasiconvex function]{\label{fig:unquasi}\includegraphics[height=4cm]{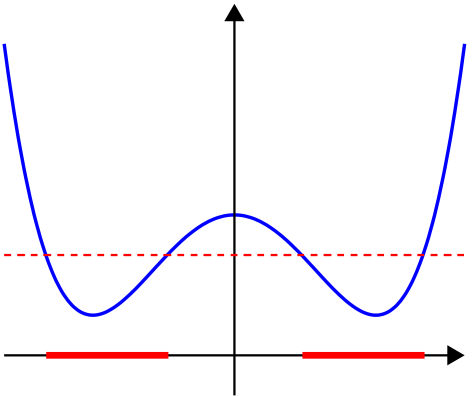}}
\caption{}
\label{quasiconvex}
\end{figure}


The convergence of the GS method is studied under different assumptions (see e.g. \cite{AAA} \cite{DPB} \cite{LM} \cite{LGM}).

\begin{theorem}[see \cite{LM}]\label{quasiconvexthm}
Suppose that the function $f$ in (\ref{GS}) is strictly quasiconvex with respect to $\bold{x_i}$ on $X$, for each $i=1, \dots, m-2$ in the sense of Definition \ref{quasi} and that the sequence $\{\bold{x}^k\}$ generated by the GS method has limit points. Then, every limit point $\overline{x}$ of $\{\bold{x}^k\}$ is a critical point of (\ref{GS}).
\end{theorem}

\begin{theorem}[see \cite{DPB}]
Let $f$ be the function in (\ref{GS}). Suppose that for each $i$ and $\bold{x} \in X$, the minimum of
$$\displaystyle\min_{\bold{\xi} \in X_i}f(\bold{x_1}, \dots, \bold{x_{i-1}}, \bold{\xi}, \bold{x_{i+1}}, \dots, \bold{x_m})$$
is uniquely attained. If $\bold{x}^k$ is the sequence generated by GS, then every limit point of $\bold{x}^k$ is a critical point.
\end{theorem}

These theorems show that the GS method can produce a converging sequence with limit points that are critical points of the problem. But, in general, the GS method may not converge, in the sense that it may produce a sequence with limit points that are not critical points of the problem. A counterexample of Powell \cite{MJD} (see also \cite{LM}) shows that for a non-convex function that is component-wise convex but not strictly quasiconvex with respect to each component, its limit points need not be critical points.

Comparing these convergence results for the GS method with the least-squares cost functionals, we observe that neither of the hypothesis in \cite{DPB} or \cite{LM} are satisfied. Indeed, the least-squares cost functional is convex (even quadratic) in each component and therefore, quasiconvex. However, in the case that the  Kathri-Rao product of two factor matrices involved is rank deficient, then the least-squares function will not be strictly quasiconvex (see the following proposition). 

\begin{proposition}
Let $f(x)=\Vert \bold{A}\bold{x}-\bold{b}\Vert^2$ where $\bold{A} \in \mathbb{R}^{m \times n}$, $m>n$, $\bold{x} \in \mathbb{R}^{n \times 1}$ and $\bold{b} \in \mathbb{R}^{m \times 1}$. If $\bold{A}$ is rank deficient, then $f(\bold{x})$ is not strictly convex.
\end{proposition}
\begin{proof}
Since $\bold{A}$ is rank deficient, then assume $rank(\bold{A})=r$ which implies that $dim(Nul \bold{A})=n-r$.  Take $\bold{x}, \bold{\tilde{x}} \in Nul \bold{A}$ where $\bold{x}\ne\bold{\tilde{x}}$. Then, according to the definition of a strictly convex function, for any $t \in [0,1]$, $f(t\bold{x}+(1-t)\bold{\widetilde{x}})= \Vert\bold{A}[t\bold{x}+(1-t)\bold{\widetilde{x}}]-\bold{b}\Vert^2 =\Vert \bold{b} \Vert^2$ and
$tf(\bold{x})+(1-t)f(\bold{\widetilde{x}})=\Vert \bold{b} \Vert^2$. Thus, $f(t\bold{x}+(1-t)\bold{\widetilde{x}})= tf(\bold{x})+(1-t)f(\bold{\widetilde{x}})$. 
\end{proof}

It follows from the proposition above that $f$ is not a strictly quasiconvex function since $f(t\bold{x}+(1-t)\bold{\widetilde{x}})=f(\bold{x})=f(\bold{\widetilde{x}})$. Thus from Theorem \ref{quasiconvexthm}, a limit point of the ALS sequence is not guaranteed to be a critical point. 

The main difficulty in proving the convergence is the lack of strict (quasi)convexity in the case of rank deficient Khatri-Rao products. This indicates that one reason for the occurrence of swamps, namely if the Khatri-Rao products of at
least two of the three iteration matrices is almost singular, the associated
least-squares functional will be flat. Thus, we can expect slow convergence, verified by the plots in  Figures \ref{fig:fifth-order-tensor} and \ref{fig: singularity}.
We observe that the region of a swamp in the ALS method (the plateau in the left convergence plot) is strongly correlated with very small singular value of the Khatri-Rao product of $\bold{B}^k$ and $\bold{C}^k$.

\begin{figure}[here]
\centering
\subfloat[Fifth order tensor]{\label{fig:fifth-order-tensor}\includegraphics[height=4cm]{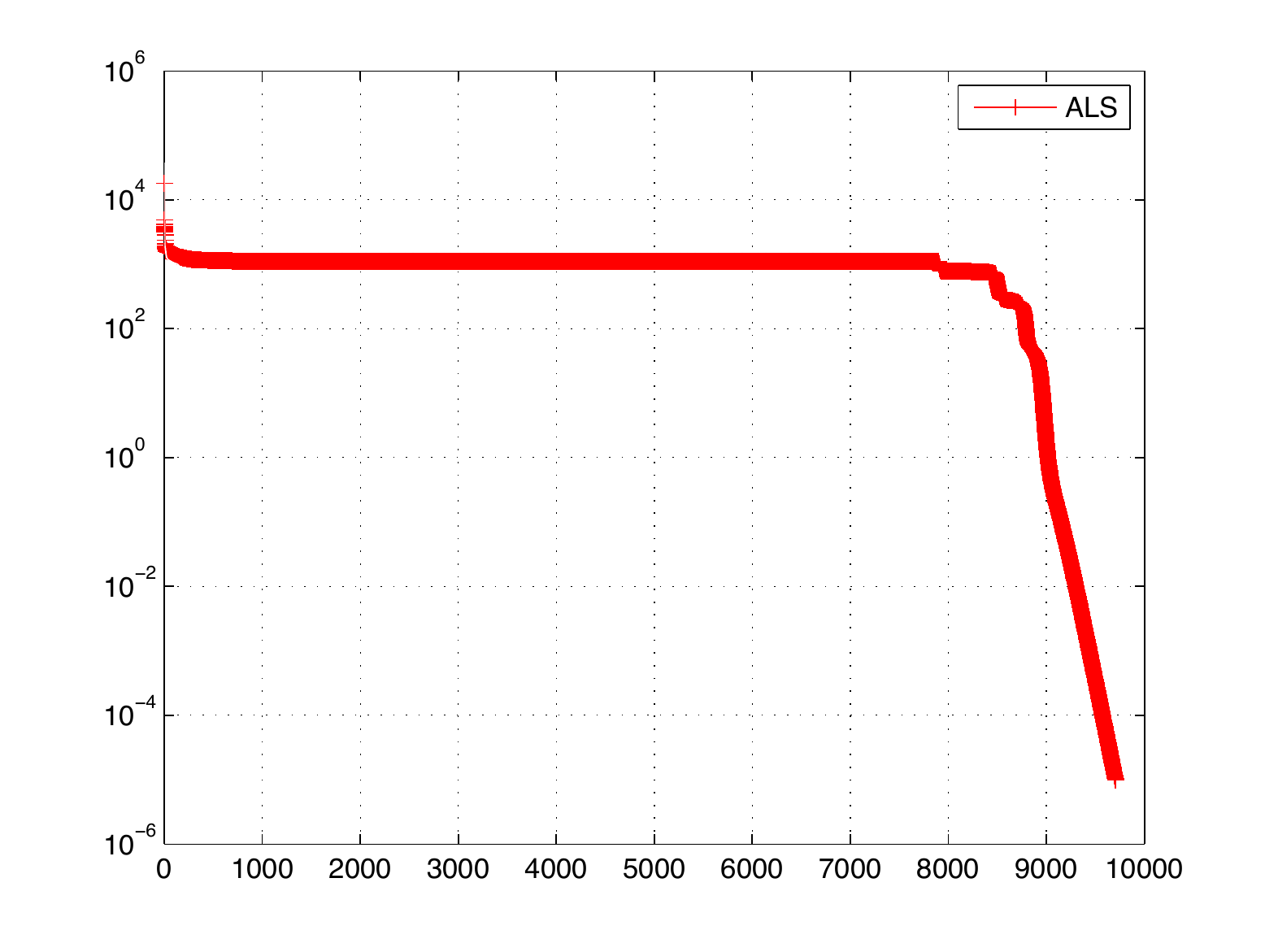}}
\qquad\qquad
\subfloat[Smallest singular values of $\bold{C}^k \odot \bold{B}^k$]{\label{fig: singularity}\includegraphics[height=4cm]{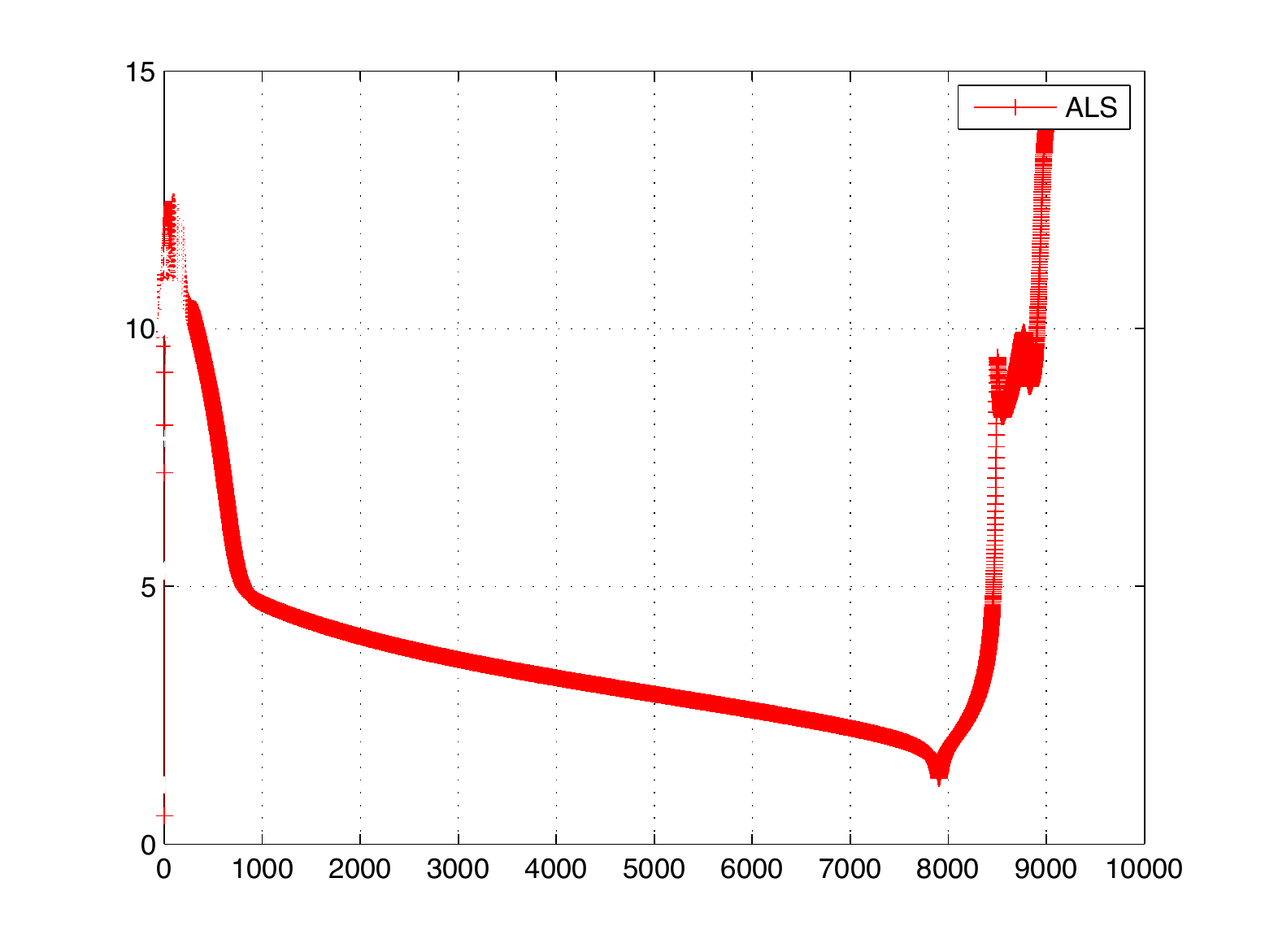}}
\caption{}
\label{fig:rank-defficient}
\end{figure}


\section{The Regularized Alternating Least-Squares}\label{RALS}


In the last section, we analyzed why ALS do not always converge through the properties of the GS method while examining RALS \cite{NDLK}, a proximal point modification of the Gauss-Seidel method (PGS) (see \cite{DPB} \cite{LM}) for tensor decomposition. The analysis provides some explanations on why RALS performs better than than ALS and decreases the high number of ALS iterations if there are swamp occurrences.

\subsection{Regularized ALS}

The regularized ALS solves the same problem (\ref{alsprob}). It recasts the main problem to three subproblems for a third-order tensor. But RALS has an extra term in each subproblem. Therefore, in order to solve the problem: 
\begin{eqnarray}\label{alsprob2}
\displaystyle\mathop{\mathrm{min}}_{\bold{A},\bold{B},\bold{C}} \quad \Vert\mathcal{T}-\sum_{r=1}^{R}\bold{a_r} \circ \bold{b_r} \circ \bold{c_r}\Vert_F^2,
\end{eqnarray}
with respect to factor matrices $\bold{A}$, $\bold{B}$ and $\bold{C}$, for a given third-order tensor $\mathcal{T}$, here are RALS subproblems:
\begin{eqnarray}\label{rals}
\bold{A}^{k+1}&=&\displaystyle\mathop{\mathrm{argmin}}_{\widehat{\bold{A}}\in \mathbb{R}^{I\times R}}\Vert\bold{T_{(1)}}^{I\times JK}-\widehat{\bold{A}}  (\bold{C}^{k}\odot\bold{B}^{k})^\text{T}\Vert_F^2 + \lambda_{k}\Vert\bold{A}^{k}-\widehat{\bold{A}}\Vert_F^2,\notag \\
\bold{B}^{k+1}&=&\displaystyle\mathop{\mathrm{argmin}}_{\widehat{\bold{B}}\in \mathbb{R}^{J\times R}}\Vert\bold{T_{(2)}}^{J\times IK}-\widehat{\bold{B}}(\bold{C}^{k}\odot\bold{A}^{k+1})^\text{T}\Vert_F^2 + \lambda_{k}\Vert\bold{B}^{k}-\widehat{\bold{B}}\Vert_F^2,\\
\bold{C}^{k+1}&=&\displaystyle\mathop{\mathrm{argmin}}_{\widehat{\bold{C}}\in \mathbb{R}^{K\times R}}\Vert\bold{T_{(3)}}^{K\times IJ}-\widehat{\bold{C}}(\bold{B}^{k+1}\odot\bold{A}^{k+1})^\text{T}\Vert_F^2 + \lambda_{k}\Vert\bold{C}^{k}-\widehat{\bold{C}}\Vert_F^2, \notag
\end{eqnarray}
where $\lambda_k >0$ is the regularization parameter. The regularization terms $\lambda_{k}\Vert\bold{A}^{k}-\widehat{\bold{A}}\Vert_F^2$, $\lambda_{k}\Vert\bold{B}^{k}-\widehat{\bold{B}}\Vert_F^2$ and $\lambda_{k}\Vert\bold{C}^{k}-\widehat{\bold{C}}\Vert_F^2$ are the fitting terms for the factors $\bold{A}$, $\bold{B}$ and $\bold{C}$.

In fact, RALS also gives us three least-squares subproblems. For example, the first subproblem in (\ref{rals}) actually is equivalent to solving a least-squares problem: 
\begin{eqnarray}\label{ralslsq}
\begin{bmatrix} 
 (\widetilde{\bold{C}}^{k}\odot\widetilde{\bold{B}}^{k}) \\
 \lambda_{k} \cdot \bold{I}^{R\times R}
\end{bmatrix}\bold{X}=
\begin{bmatrix}
{\bold{T_{(1)}}}^{\text{T}} \\
\lambda_{k} \cdot (\widetilde{\bold{A}}^k)^{\text{T}}
\end{bmatrix},
\end{eqnarray}
which is different from the least-squares obtained from ALS, that is,
\begin{eqnarray}\label{alslsq}
(\bold{C}^{k}\odot\bold{B}^{k})\bold{X}={\bold{T_{(1)}}}^{\text{T}}.
\end{eqnarray}

\begin{framed}
  \textbf{\textsf{RALS-Algorithm}}
  
\textbf{procedure} CP-RALS($\mathcal{X}, R, N, \lambda$)

\hskip10pt give initial guess $\bold{A}^0 \in \mathbb{R}^{I \times R}$, $\bold{B}^0 \in \mathbb{R}^{J \times R}$, $\bold{C}^0\in \mathbb{R}^{K \times R}$, $\lambda_0$

\hskip10pt \textbf{for} $n=1,\dots, N$ \textbf{do}

\hskip25pt $\bold{W} \leftarrow [(\bold{C}^{n}\odot \bold{B}^{n}) ;  \lambda_n \bold{I}^{R\times R}] \in \mathbb{R}^{(JK+R)\times R}$ 
\vskip5pt
\hskip27pt $\bold{S} \leftarrow [\bold{X_{(1)}}^{\text{T}} ;  \lambda_n (\bold{A}^n)^{\text{T}}] \in \mathbb{R}^{(JK+R)\times I}$
\vskip5pt
\hskip26pt $\bold{A}^{n+1} \leftarrow \bold{W}/\bold{S}$ ------ \% solving least squares to update $\bold{A}$
\vskip5pt

\hskip25pt $\bold{W} \leftarrow [(\bold{C}^{n}\odot \bold{A}^{n+1}) ;  \lambda_n\bold{I}^{R\times R}] \in \mathbb{R}^{(IK+R)\times R}$ 
\vskip5pt
\hskip27pt $\bold{S} \leftarrow [\bold{X_{(2)}}^{\text{T}} ;  \lambda_n(\bold{B}^n)^{\text{T}}] \in \mathbb{R}^{(IK+R)\times J}$
\vskip5pt
\hskip26pt $\bold{B}^{n+1} \leftarrow \bold{W}/\bold{S}$ ------ \% solving least squares to update $\bold{B}$
\vskip5pt
\hskip25pt $\bold{W} \leftarrow [(\bold{B}^{n+1}\odot \bold{A}^{n+1}) ;  \lambda_n\bold{I}^{R\times R}] \in \mathbb{R}^{(IJ+R)\times R}$ 
\vskip5pt
\hskip27pt $\bold{S} \leftarrow [\bold{X_{(3)}}^{\text{T}} ;  \lambda_n(\bold{C}^n)^{\text{T}}] \in \mathbb{R}^{(IJ+R)\times K}$
\vskip5pt
\hskip26pt $\bold{C}^{n+1} \leftarrow \bold{W}/\bold{S}$ ------ \% solving least squares to update $\bold{C}$
\vskip5pt
\hskip26pt $\lambda_{n+1} \leftarrow  \delta \cdot \lambda_n$ ------ \% update regularization parameter
\vskip5pt
\hskip10pt \textbf{end for}

\hskip10pt return $\bold{A}^{N}$, $\bold{B}^{N}$, $\bold{C}^{N}$

\textbf{end procedure}
\end{framed}
The number of iteration $N$ is set to a large number; otherwise a convergence stopping criterion can be used.

Our notion of a regularized ALS can be misleading. In the usual regularization setting, the minimizer (critical point) of the regularized cost functional is found. The additional regularization terms in \ref{rals} penalize the difference between the previous iterates, which themselves need not be a bounded sequence. Although there is regularization in each step, the method imposes no \emph{uniform} constraint for all $k$ on the matrices $\bold{A}^k,\bold{B}^k,\bold{C}^k$.  For this reason, the cost functional \ref{rals} is not a constrained optimization problem, i.e. no global optimal solutions are guaranteed.  In particular, this approach does not address the degeneracy problem. Moreover, the limit points of the RALS algorithm will be shown as the critical 
points of the least-squares functional \ref{alsprob2} and
not of the regularized version.


The RALS method differs from the Tikhonov regularization for CP decomposition considered in \cite{LimComon} \cite{PP} in the following way:
the Tikhonov functional minimized is
\begin{equation}\label{rfunct} 
\left \Vert\mathcal{T}-\sum_{r=1}^{R}\bold{a_r} \circ \bold{b_r} \circ \bold{c_r} \right \Vert_F^2 + 
 \lambda\left( \Vert \bold{A} \Vert_F^2 + \Vert \bold{B} \Vert_F^2 + \Vert \bold{C} \Vert_F^2 \right).
\end{equation}
If ALS is applied to this regularized functional, then the corresponding subproblems are
\begin{eqnarray}\label{alsr}
\bold{A}^{k+1}&=&\displaystyle\mathop{\mathrm{argmin}}_{\widehat{\bold{A}}\in \mathbb{R}^{I\times R}}\Vert\bold{T_{(1)}}^{I\times JK}-\widehat{\bold{A}}  (\bold{C}^{k}\odot\bold{B}^{k})^\text{T}\Vert_F^2 + \lambda \Vert \widehat{\bold{A}}\Vert_F^2,\notag \\
\bold{B}^{k+1}&=&\displaystyle\mathop{\mathrm{argmin}}_{\widehat{\bold{B}}\in \mathbb{R}^{J\times R}}\Vert\bold{T_{(2)}}^{J\times IK}-\widehat{\bold{B}}(\bold{C}^{k}\odot\bold{A}^{k+1})^\text{T}\Vert_F^2 + \lambda \Vert \widehat{\bold{B}}\Vert_F^2,\\
\bold{C}^{k+1}&=&\displaystyle\mathop{\mathrm{argmin}}_{\widehat{\bold{C}}\in \mathbb{R}^{K\times R}}\Vert\bold{T_{(3)}}^{K\times IJ}-\widehat{\bold{C}}(\bold{B}^{k+1}\odot\bold{A}^{k+1})^\text{T}\Vert_F^2 + \lambda \Vert\widehat{\bold{C}}\Vert_F^2. \notag
\end{eqnarray}
Observe that the penalization terms, $\Vert \widehat{\bold{A}}\Vert_F^2$, $\Vert \widehat{\bold{B}}\Vert_F^2$ and $\Vert   \widehat{\bold{C}}\Vert_F^2$, are independent of $k$, which are viewed as uniform constraints on the norm of the matrices. 
From \cite{LimComon}, this constrained optimization problem \ref{rfunct} always has a globally optimal solution.
However, the price to pay here is that the optimal solution is not a critical point of the 
\ref{alsprob2}, but it is a critical point of the regularized functional.

\subsubsection{Proximal Point Modification of the Gauss-Seidel (PGS) Method}

In Section $3$, we have shown that the ALS (GS) method may not converge due to a requirement of convexity or quasiconvexity assumption to guarantee convergence. Thus, a modification of GS is considered by adding an extra term in each iteration:
$$\bold{x_i}^{k+1}=\displaystyle\mathop{\mathrm{argmin}}_{\bold{y_i} \in X_i} f(\bold{x_1}^{k+1}, \dots, \bold{y_i}, \dots, \bold{x_m}^{k}) + \frac{1}{2} \tau_i \Vert \bold{y_i} - \bold{x_i}^k\Vert^2.$$
This method is called the proximal point modification of the GS  (PGS) method (see \cite{DPB}, \cite{LM}). It is also referred as partial proximal minimization \cite{DPP}. The PGS formulation lead to a weakened assumption for convergence to critical points.

\begin{definition}
The GS and PGS methods are well-defined if every subproblem has solutions. 
\end{definition}

\begin{proposition}[Convergence proposition of PGS \cite{LM}]
Suppose that the PGS method is well defined and that the sequence $\{\bold{x}^k\}$ has limit points, then every limit point $\overline{\bold{x}}$ of $\{\bold{x}^k\}$ is a critical point of problem (\ref{GS}).
\end{proposition} 

Recall that in Section \ref{bngsm}, we showed that the ALS method is the GS method for CP decomposition with respect to the factor matrices $\bold{A}$, $\bold{B}$ and $\bold{C}$. Similarly, through vectorization of the three factor matrices, we have
\begin{eqnarray*}
&\bold{x_{1}}^{k+1}&=\displaystyle\mathop{\mathrm{argmin}}_{\bold{y_1} \in \mathbb{R}^{n_1}}\{f(\bold{y_1}, \bold{x_{2}}^{k}, \bold{x_{3}}^{k})+\lambda_{k}\Vert\bold{x_{1}}^{k}-\bold{y_1}\Vert_F^2 \},\\
&\bold{x_{2}}^{k+1}&=\displaystyle\mathop{\mathrm{argmin}}_{\bold{y_2} \in \mathbb{R}^{n_2}}\{f(\bold{x_{1}}^{k+1}, \bold{y_2}, \bold{x_{3}}^{k})+\lambda_{k}\Vert\bold{x_{2}}^{k}-\bold{y_2}\Vert_F^2\},\\
&\bold{x_{3}}^{k+1}&=\displaystyle\mathop{\mathrm{argmin}}_{\bold{y_3} \in \mathbb{R}^{n_3}}\{f(\bold{x_{1}}^{k+1}, \bold{x_{2}}^{k+1}, \bold{y_3})+\lambda_{k}\Vert\bold{x_{3}}^{k}-\bold{y_3}\Vert_F^2\}.
\end{eqnarray*}
Thus the regularized ALS is the PGS method for CP decomposition.

\subsection{Convergence Result of the Regularized ALS}\label{sec: main}

In this section, we will show that the converging sequence obtained from RALS method leads to a critical point. This characterization is not true for the ALS algorithm as we have seen in Section \ref{alswithgs} where converging sequence of factor matrices $\{(\bold{A}^k, \bold{B}^k, \bold{C}^k)\}$  cannot guarantee that the limit point is a critical point (local minimum). 

We adapt the proposition in Section 7 in \cite{LM} to our problem.

\begin{thm}\label{conv}
Suppose that the sequence $\{(\bold{A}^k, \bold{B}^k, \bold{C}^k)\}$ obtained from RALS has limit points, then every limit point $(\overline{\bold{A}}, \overline{\bold{B}}, \overline{\bold{C}})$ is a critical point of the Problem (\ref{alsprob2}).
\end{thm}
\begin{proof}
Recall the vectorization in Section \ref{alswithgs} which allows us to re-express $\{(\bold{A}^k, \bold{B}^k, \bold{C}^k)\}$ as $(\bold{x_1}, \bold{x_2},\bold{x_3})$ and the cost function as
\[f(\bold{x_1}, \bold{x_2}, \bold{x_3})=\displaystyle\sum_{k=1}^{K}\sum_{j=1}^{J}\sum_{i=1}^{I}(t_{ijk}-\sum_{r=1}^{R}a_{ir}b_{jr}c_{kr})^2\]
where $\bold{x_1}=vec(\bold{A}) \in \mathbb{R}^{IR}$, $\bold{x_2}=vec(\bold{B}) \in \mathbb{R}^{JR}$ and $\bold{x_3}=vec(\bold{C}) \in \mathbb{R}^{KR}$. Let $\{\bold{x}^{n_k}\}_{k=1}^{\infty}=\{(\bold{x_1}^{n_k}, \bold{x_2}^{n_k}, \bold{x_3}^{n_k})\}_{k=1}^{\infty}$ be the converging subsequence of $\{(\bold{x_1}^k, \bold{x_2}^k, \bold{x_3}^k)\}$ which has the limit point $(\bold{\overline{x}_1}, \bold{\overline{x}_2}, \bold{\overline{x}_3})$.

The subproblem in the RALS method provides the following inequality:
\begin{eqnarray}\label{limit}
f(\bold{x_1}^{n_{k}+1}, \bold{x_2}^{n_k}, \bold{x_3}^{n_k}) \leq f(\bold{x_1}^{n_{k}}, \bold{x_2}^{n_k}, \bold{x_3}^{n_k}) - \lambda_{n_k}\Vert\bold{x_1}^{n_k+1}-\bold{x_1}^{n_k}\Vert^2.
\end{eqnarray}
Using the inequality above repeatedly, we have 
\begin{eqnarray}\label{limitit}
f(\bold{x_1}^{n_k+1}, \bold{x_2}^{n_k+1}, \bold{x_3}^{n_k+1}) 
&\leq& f(\bold{x_1}^{n_k+1}, \bold{x_2}^{n_k+1}, \bold{x_3}^{n_k}) \notag \\ 
&\leq& f(\bold{x_1}^{n_k+1}, \bold{x_2}^{n_k}, \bold{x_3}^{n_k})  \\
&\leq& f(\bold{x_1}^{n_k}, \bold{x_2}^{n_k}, \bold{x_3}^{n_k}). \notag
\end{eqnarray}
By the Squeeze Theorem, the continuity of $f$ and $(\bold{x_1}^{n_k}, \bold{x_2}^{n_k}, \bold{x_3}^{n_k}) \longrightarrow (\bold{\overline{x}_1}, \bold{\overline{x}_2}, \bold{\overline{x}_3})$ as $k \rightarrow \infty$, then we have the following
$$\displaystyle\lim_{k\rightarrow \infty}f(\bold{x_1}^{n_k+1}, \bold{x_2}^{n_k}, \bold{x_3}^{n_k})=\displaystyle\lim_{k \rightarrow \infty}f(\bold{x_1}^{n_k}, \bold{x_2}^{n_k}, \bold{x_3}^{n_k})= f(\bold{\overline{x}_1}, \bold{\overline{x}_2}, \bold{\overline{x}_3}).$$
Now taking the limits in (\ref{limit}) for $k \rightarrow \infty$ on both sides, we have
\begin{eqnarray}\label{proof1}
\displaystyle\lim_{k \rightarrow \infty} \Vert\bold{x_1}^{n_k+1}-\bold{x_1}^{n_k}\Vert^2 =0
\end{eqnarray}
which implies
\begin{eqnarray}\label{proof2}
\displaystyle\lim_{k \rightarrow \infty} (\bold{x_1}^{n_k+1}, \bold{x_2}^{n_k}, \bold{x_3}^{n_k})=(\bold{\overline{x}_1}, \bold{\overline{x}_2}, \bold{\overline{x}_3}).
\end{eqnarray}
Similarly, we obtain
\begin{eqnarray}\label{proof3}
\displaystyle\lim_{k \rightarrow \infty} (\bold{x_1}^{n_k+1}, \bold{x_2}^{n_k+1}, \bold{x_3}^{n_k})=(\bold{\overline{x}_1}, \bold{\overline{x}_2}, \bold{\overline{x}_3}). 
\end{eqnarray}

Since every RALS subproblem is well defined, then each point in the subsequence satisfies the corresponding optimality condition (Theorem \ref{opt}); i.e.
\begin{eqnarray}
&&\triangledown_{1}f(\bold{x_1}^{n_k+1}, \bold{x_2}^{n_k}, \bold{x_3}^{n_k}) + 2\lambda_{n_k}(\bold{x_1}^{n_k+1}-\bold{x_1}^{n_k}) = 0, \label{oc1}\\
&&\triangledown_{2}f(\bold{x_1}^{n_k+1}, \bold{x_2}^{n_k+1}, \bold{x_3}^{n_k}) + 2\lambda_{n_k}(\bold{x_2}^{n_k+1}-\bold{x_2}^{n_k}) = 0, \label{oc2}\\
&&\triangledown_{3}f(\bold{x_1}^{n_k+1}, \bold{x_2}^{n_k+1}, \bold{x_3}^{n_k+1}) + 2\lambda_{n_k}(\bold{x_3}^{n_k+1}-\bold{x_3}^{n_k}) = 0. \label{oc3}
\end{eqnarray}
Then, taking $k \rightarrow \infty$ in (\ref{oc1}--\ref{oc3}), using the arguments in (\ref{proof1}), (\ref{proof2}), (\ref{proof3}) and the continuity of $\triangledown f$, we obtain
$$\triangledown_{i}f(\bold{\overline{x}_1}, \bold{\overline{x}_2}, \bold{\overline{x}_3})=0,\quad i=1,2,3.$$

Thus, this proves that the limit point $(\bold{\overline{x}_1}, \bold{\overline{x}_2}, \bold{\overline{x}_3})$ is a critical point of the cost function $f(\bold{x_1}, \bold{x_2}, \bold{x_3})$. Furthermore, we obtain
\begin{eqnarray}\label{ocnormal}
&&\triangledown_{\bold{A}}f(\bold{\overline{A}}, \bold{\overline{B}}, \bold{\overline{C}})=0, \\
&&\triangledown_{\bold{B}}f(\bold{\overline{A}}, \bold{\overline{B}}, \bold{\overline{C}})=0, \nonumber\\
&&\triangledown_{\bold{C}}f(\bold{\overline{A}}, \bold{\overline{B}}, \bold{\overline{C}})=0. \nonumber
\end{eqnarray}
through the inverse mapping of the vectorization. Therefore, $(\bold{\overline{A}}, \bold{\overline{B}}, \bold{\overline{C}})$ is a critical point.
\end{proof}

Here are some remarks:
\begin{enumerate}
\item
Following from the discussion and the theorem above, we showed that RALS method solves the same ALS cost function.  Moreover, we have proved that the limit point obtained from RALS is a critical point of the original minimization problem of $\Vert\mathcal{T}-\widehat{\mathcal{T}}\Vert_F^2 $.
\item
The main theorem above solves the CP decomposition on the whole space, so we use the optimality condition, $\triangledown f (\bold{\overline{x}_1}, \bold{\overline{x}_2}, \bold{\overline{x}_3})=0$. If the solution is not in the whole space, namely, in the problem of  non-negative tensor decomposition, then the optimality condition,$\triangledown f (\bold{\overline{x}_1}, \bold{\overline{x}_2}, \bold{\overline{x}_3})^{\text{T}}(\bold{y}-\bold{\overline{x}_i}) \geq 0$, must be used. 
\item
For the ALS method, under the same assumption in Theorem \ref{conv}, the theorem may not be true. From the assumption, we know that the sequence $\{(\bold{A}^{n_k}, \bold{B}^{n_k}, \bold{C}^{n_k})\}$ converges to a limit point $(\overline{\bold{A}}, \overline{\bold{B}}, \overline{\bold{C}})$, but we cannot obtain the sequences $\{(\bold{A}^{n_k+1}, \bold{B}^{n_k}, \bold{C}^{n_k})\}$ and $\{(\bold{A}^{n_k+1}, \bold{B}^{n_k+1}, \bold{C}^{n_k})\}$ to converge. Furthermore, we also cannot prove that these two sequences converge to the same limit point $(\overline{\bold{A}}, \overline{\bold{B}}, \overline{\bold{C}})$.  
\item
The optimality conditions in (\ref{ocnormal}) are equivalent to the normal equations of the subproblems:
\begin{eqnarray*}
& &\bold{T_{(1)}}^{I\times JK}(\bold{\overline{C}}\odot \bold{\overline{B}})= \bold{\overline{A}}  ( \bold{\overline{C}}\odot \bold{\overline{B}})^{\text{T}}( \bold{\overline{C}}\odot \bold{\overline{B}}),\\
& &\bold{T_{(2)}}^{J\times IK}(\bold{\overline{C}}\odot \bold{\overline{A}})= \bold{\overline{B}}  ( \bold{\overline{C}}\odot \bold{\overline{A}})^{\text{T}}( \bold{\overline{C}}\odot \bold{\overline{A}}),\\
& &\bold{T_{(3)}}^{K\times IJ}(\bold{\overline{B}}\odot \bold{\overline{A}})= \bold{\overline{C}}  ( \bold{\overline{B}}\odot \bold{\overline{A}})^{\text{T}}( \bold{\overline{B}}\odot \bold{\overline{A}}).
\end{eqnarray*}


\item 
Theorem \ref{conv} is a \emph{conditional} convergence proof, impinging upon the existence of the ALS limit points.
Thus this result does not address the degeneracy problems. Analysis of the existence of the limits of the (R)ALS is a challenging problem that would require a careful study of the degenerate cases of the CP decomposition.
The regularization \ref{rfunct} considered by Paatero \cite{PP} is a good approach in finding approximation to the degenerate case, but the solutions satisfy the \emph{regularized} cost functional and not the original least-squares functional.
Moreover, a similar conditional convergent analysis \cite{LM} can be established for the regularized functional \ref{rfunct}.  In fact, if $\lambda >0$, then the cost functional \ref{rfunct}
will be component-wise strictly quasiconvex. Thus Theorem~\ref{quasiconvexthm} applies and hence, the limit points
of \ref{alsr} will be critical points of the regularized functional \ref{rfunct}.

%

\end{enumerate}

\section{Numerical Comparison of the ALS and RALS Algorithms}

In this section, we compare the performance of RALS against ALS. We give three examples of third-order tensor CP decomposition to demonstrate the swamp shortening property of the iterated regularization and one example of large real third-order tensor data. 

\subsection{Example I: Initial Factors Dependent Swamp}
Let the matrices
\begin{eqnarray*}
\bold{A}=\begin{bmatrix}
1 & 2 \\
2 & 1 \\
3 & 2
\end{bmatrix}, \quad 
\bold{B}=\begin{bmatrix}
2 & 1 \\
-1 & 3 \\
1 & -1 
\end{bmatrix}, \quad
\bold{C}=\begin{bmatrix}
3 & 1\\
1 & 2 \\
2 & 2
\end{bmatrix},
\end{eqnarray*}
be the three factor matrices of a third-order tensor $\mathcal{T} \in \mathbb{R}^{3 \times 3 \times 3}$ of rank-two:
$$\mathcal{T}=\bold{a_1}\circ \bold{b_1} \circ \bold{c_1} + \bold{a_2} \circ \bold{b_2} \circ \bold{c_2}.$$


\begin{figure}[h]
\centering
\subfloat[right initial guess]{\label{fig: permutation}\includegraphics[height=4cm]{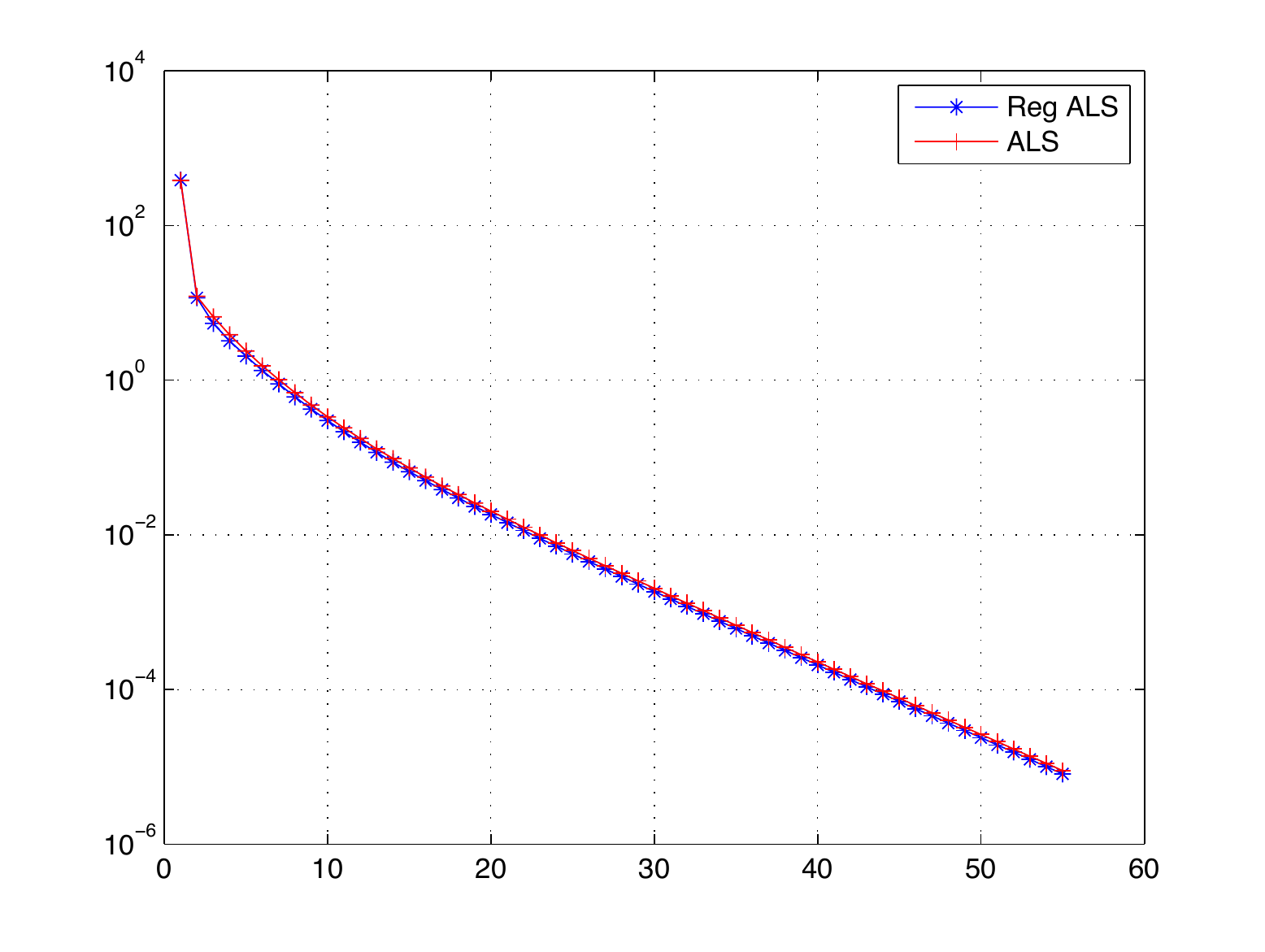}}
\qquad\quad
\subfloat[random initial guess]{\label{fig: bad}\includegraphics[height=4cm]{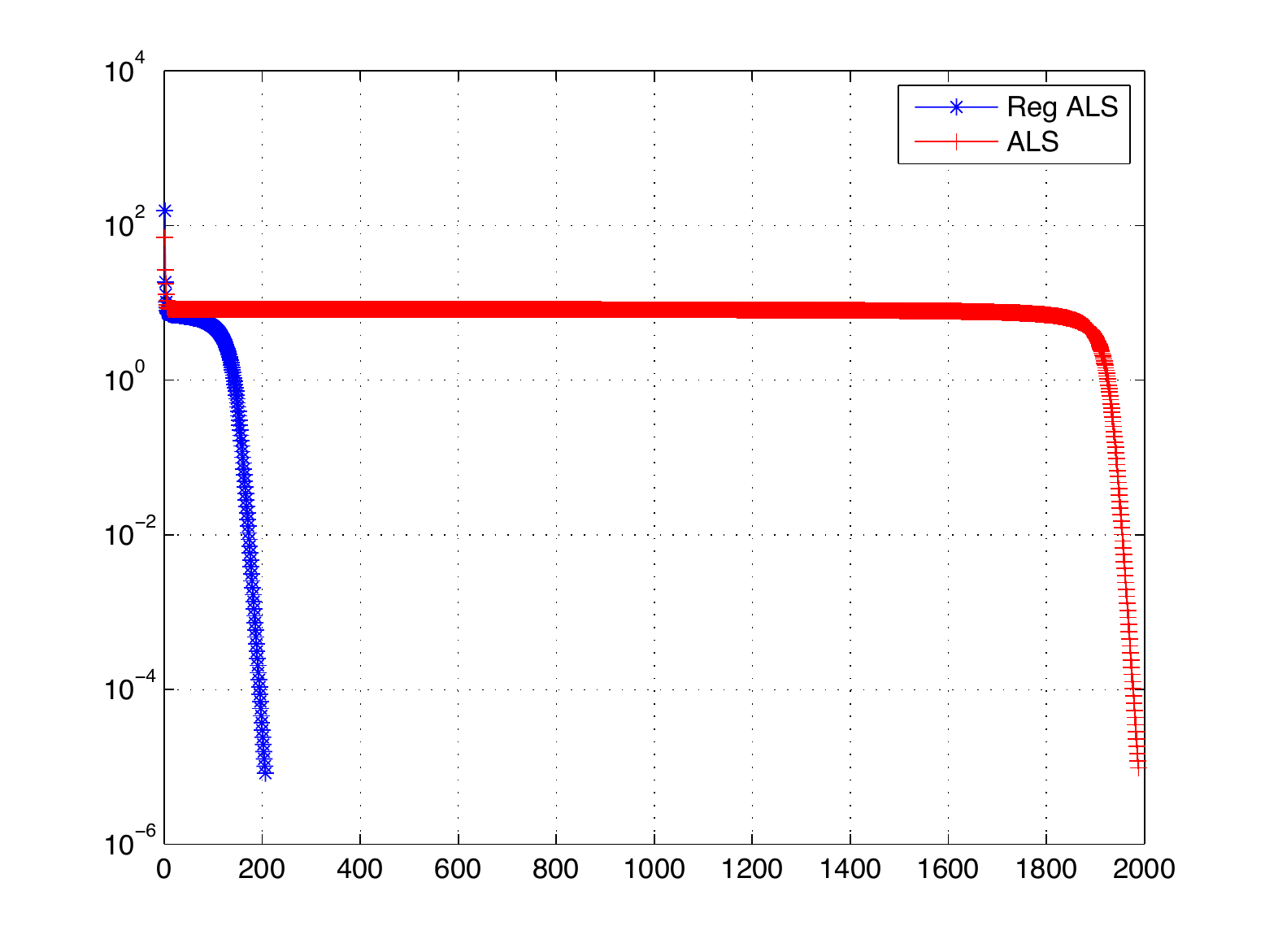}}
\caption{}
\label{fig:example1}
\end{figure}


In the two figures, the plots show the error $\Vert \mathcal{T}-\mathcal{T}_{est}\Vert_F^2$ versus the number of iterations it takes to obtain an error of $1\times 10^{-5}$, where $\mathcal{T}_{est}$ denotes the obtained tensor after every iteration. The red line denotes ALS method while the blue one is RALS in each picture for the same initial guess. 

Two initial guesses are compared in Figures \ref{fig: permutation}--\ref{fig: bad} in terms of ALS. In Figure \ref{fig: permutation}, $\bold{A}^0=\bold{A}$, $\bold{B}^0=\bold{B}\begin{bmatrix}0 & 1\\ 1 & 0\end{bmatrix}$ and $\bold{C}^0=\bold{C}$. 
For Figure \ref{fig: bad}, we randomly generated $3\times 2$ matrices as the initial factors. With $\{\bold{A}, \bold{B}\begin{bmatrix}0 & 1\\ 1 & 0\end{bmatrix}, \bold{C}\}$ as the initial guess, ALS takes $55$ iterations to reach an error within $10^{-5}$ while it takes 1988 iterations by using random initial guess.
Observe in Figure \ref{fig: permutation} that ALS and RALS have the same convergence speed and take the same iterations to reach an error within $10^{-5}$. However, in Figure \ref{fig: bad}, RALS can reduce the swamp by only taking $206$ iterations in comparison to that of  $1988$ ALS iterations. 
In some cases, randomly generated factors can lead to swamp in the implementation of the ALS. However this swamp phenomena induced by the initial factors is not observed if the RALS method is used.

\subsection{Example II: Rank Specific Swamp}
Let the matrices 
\begin{eqnarray*}
\bold{A}=\begin{bmatrix}
1 & 2 &3 \\
2 & 1 &2
\end{bmatrix}, \quad 
\bold{B}=\begin{bmatrix}
2 & 1 & 1\\
-1 & 3 &1
\end{bmatrix}, \quad
\bold{C}=\begin{bmatrix}
3 & 1 &2\\
1 & 2 &-1
\end{bmatrix}.
\end{eqnarray*}
be the three factor matrices of a third-order tensor $\mathcal{T}=\displaystyle\sum_{r=1}^{3} \bold{a_r}\circ \bold{b_r}\circ \bold{c_r}\in \mathbb{R}^{2 \times 2 \times 2}$ that is a rank-three tensor. Rank-two (Figure \ref{fig: rank-two}) and rank-three (Figure \ref{fig: rank-three}) approximations are calculated with the following initial matrices:
\begin{eqnarray*}
\bold{A}^0=\begin{bmatrix}
0.1679 &   0.7127 \\
 0.9787  &  0.5005
\end{bmatrix}, \quad 
\bold{B}^0=\begin{bmatrix}
0.4711  &  0.6820 \\
0.0596   & 0.0424
\end{bmatrix}, \quad
\bold{C}^0=\begin{bmatrix}
  0.0714  &  0.0967 \\
  0.5216  &   0.8181
\end{bmatrix}.
\end{eqnarray*}
So, the following picture shows the error plot by using ALS and RALS separately:

\begin{figure}[h]
\centering
\subfloat[Rank-three]{\label{fig: rank-three}\includegraphics[height=4cm]{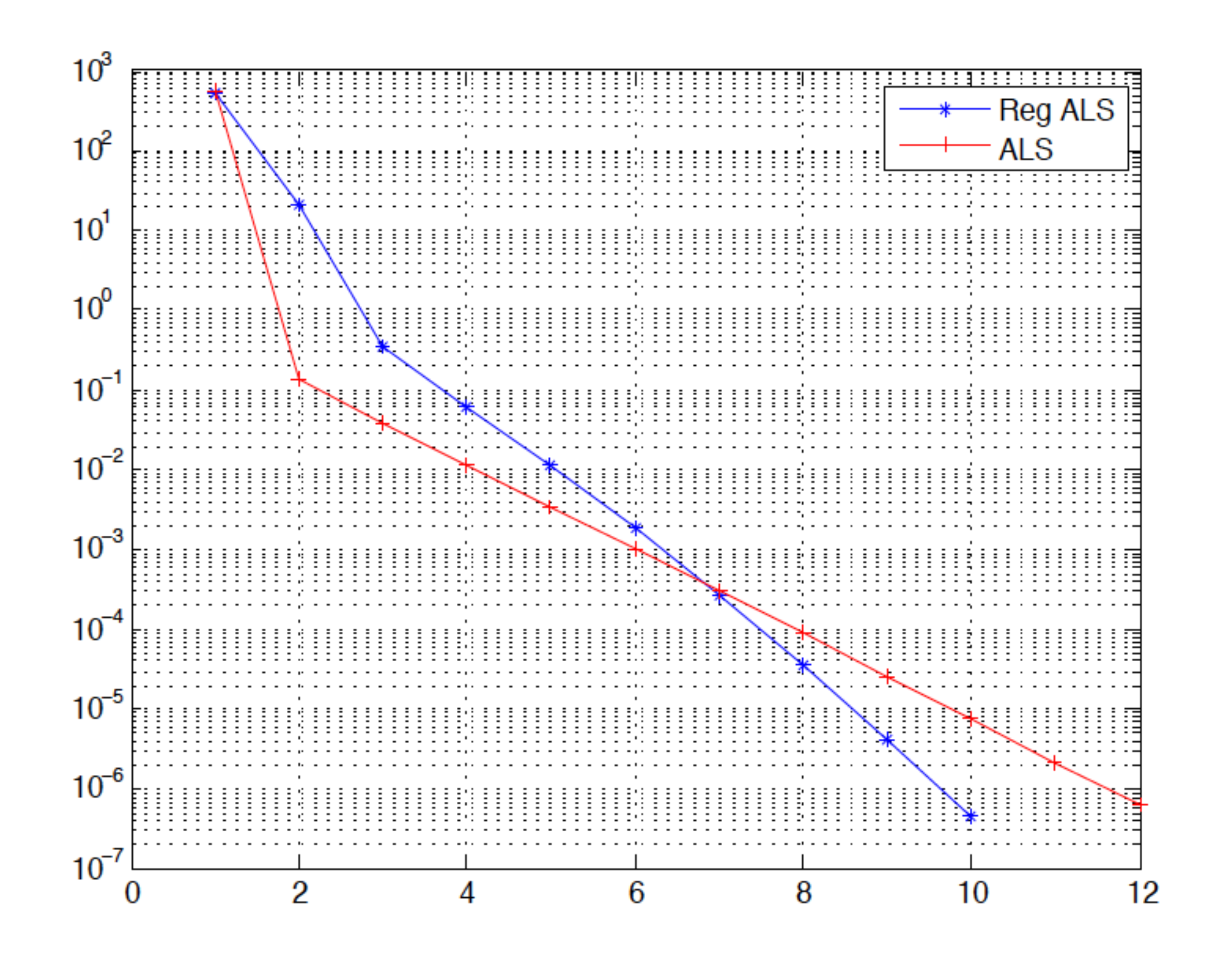}}
\qquad\quad
\subfloat[Rank-two]{\label{fig: rank-two}\includegraphics[height=4cm]{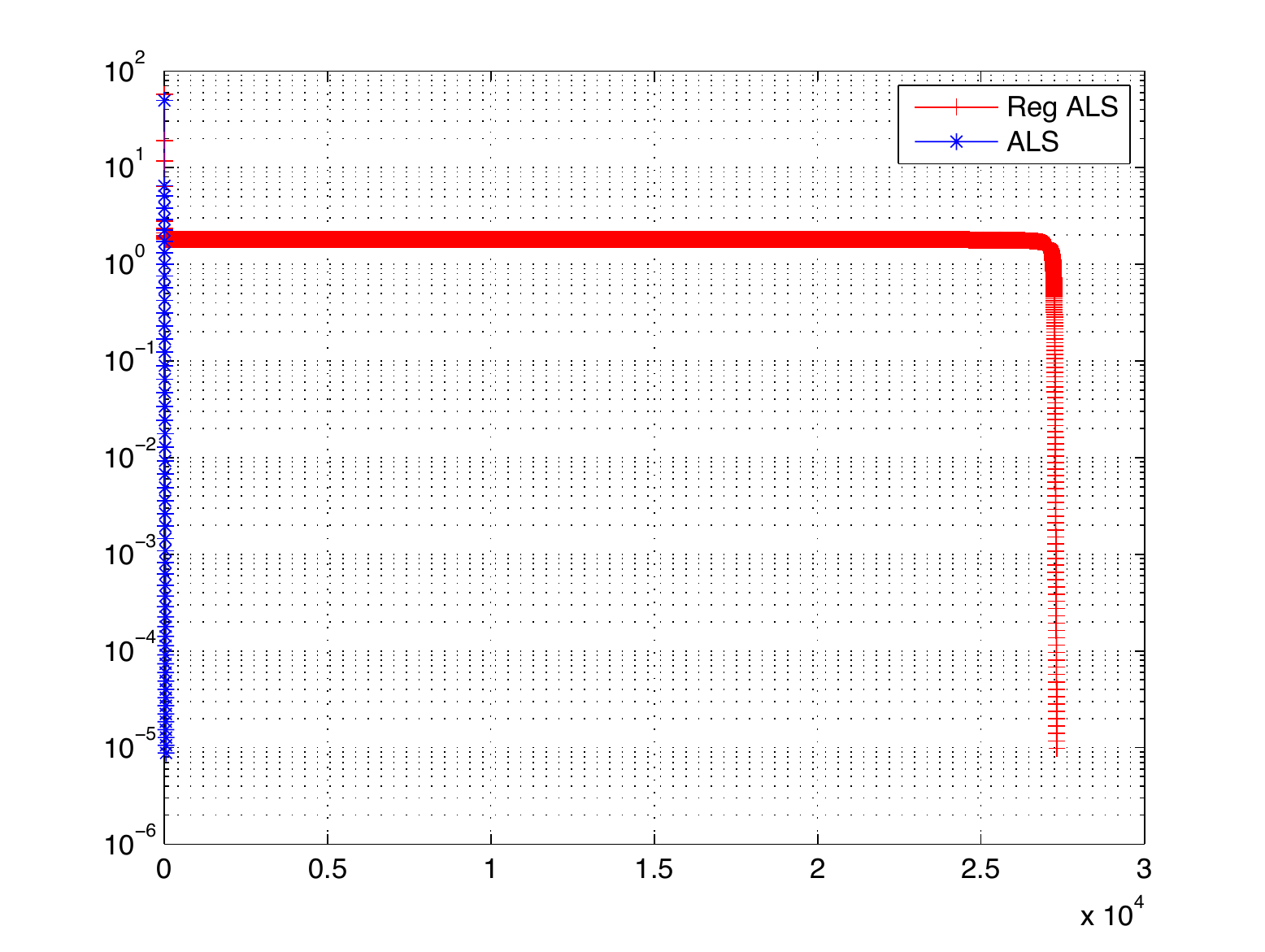}}
\caption{}
\label{fig: example2}
\end{figure}

Notice that the rank-three tensor approximations present no problem in both ALS and RALS as seen in Figure \ref{fig: rank-three}. However, in Figure \ref{fig: rank-two}, the rank-two tensor approximation requires only $53$ iterations RALS (blue line) to reach an error within $10^{-5}$ while ALS needs $27322$ iterations as indicated in a swamp. Further investigation is needed to understand the degeneracy problems with respect to the RALS algorithm.

\subsection{Example III: Induced Rank-Deficiency Swamp}
From the example in Section \ref{analysis-als}, the RALS and ALS are compared. Recall that the rank deficiency of the Khatri-Rao products induce an ALS swamp. In the Figure \ref{fig: fifth-swamp}, the error plots  show a swamp for ALS with $9707$ iterations while RALS exhibits no swamp with only $884$ iterations. 
\begin{figure}[h]
\centering
\subfloat[Fifth order tensor]{\label{fig: fifth-swamp}\includegraphics[height=4cm]{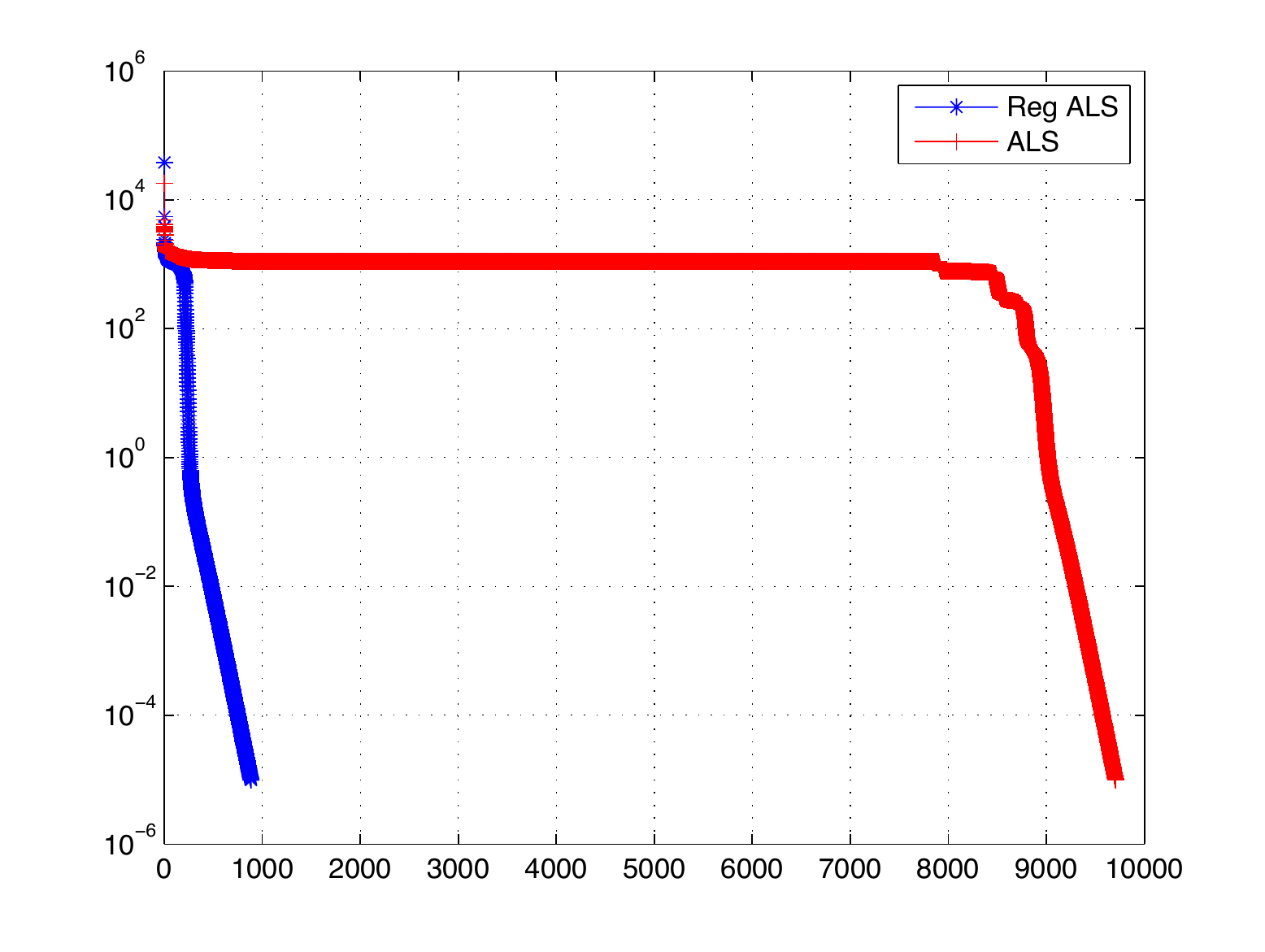}}
\qquad\quad
\subfloat[Singularity of Khatri-Rao]{\label{fig: singularity2}\includegraphics[height=4cm]{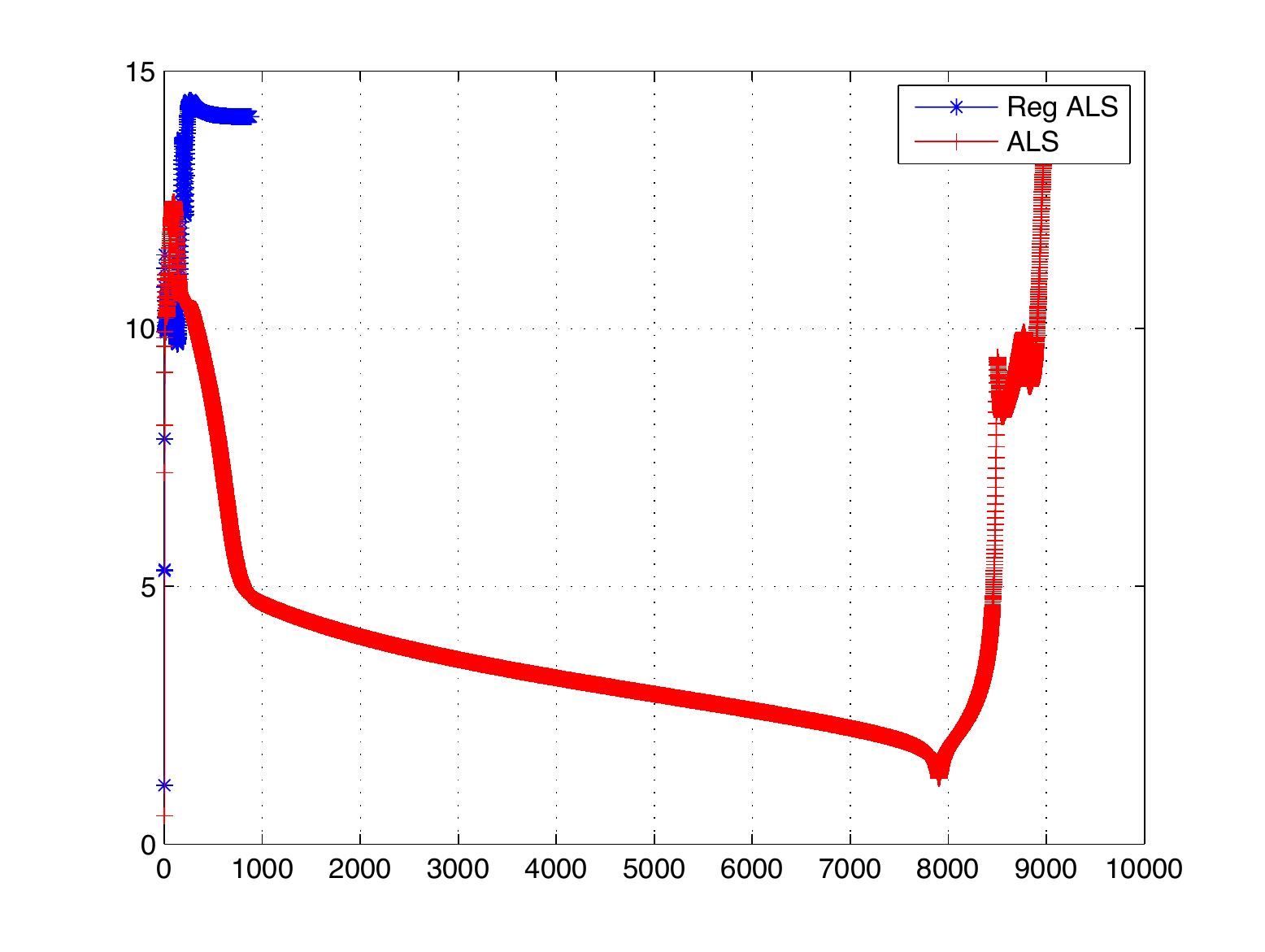}}
\caption{}
\label{fig: example3}
\end{figure}

To understand why RALS is not hampered by a swamp, let's look at the normal equation of the subproblem (we have already mentioned in the last section (\ref{ralslsq})):
\begin{eqnarray*}
\begin{bmatrix} 
 (\widetilde{\bold{C}}^{k}\odot\widetilde{\bold{B}}^{k}) \\
 \lambda_{k} \cdot \bold{I}^{R\times R}
\end{bmatrix}\bold{X}=
\begin{bmatrix}
{T_{(1)}}^{\text{T}} \\
\lambda_{k} \cdot (\widetilde{\bold{A}}^k)^{\text{T}}
\end{bmatrix}
\end{eqnarray*}
where the least squares solution is $\widetilde{\bold{A}}^{k+1}$. The submatrix $\lambda_{k} \cdot \bold{I}^{R\times R}$ in (\ref{ralslsq}) embeds the range space of $(\widetilde{\bold{C}}^{k}\odot\widetilde{\bold{B}}^{k})$ in a higher dimensional space while induces a full rank linear least-squares subproblem. Thus,  the regularization keeps the cost function strictly component-wise quasiconvex.

\subsection{Example IV: Large Real Datasets}
Since ALS type algorithms have been particularly useful in real large datasets, a comparison study of the ALS and RALS algorithms was made on a tensor $\mathcal{X} \in \mathbb{R}^{170 \times 274 \times 35}$ from the paper of Bro et al.~\cite{Bro1} in detecting and characterizing active photosensitizers in butter. The light exposure experimental data is obtained from 
 different colors of light, variation in oxygen availability, and time of exposure while measuring the fluorescence EEMs (excitation emission matrices) and sensory evaluation of the samples.
Thus the element $x_{ijk}$ represents the fluorescence intensity for sample $i$ at excitation wavelength $j$ and emission wavelength $k$. CP algorithms offer decomposition into factors of sample scores, emission loadings and excitation loadings. 

The ALS and RALS algorithms were used to analyze the fluorescence landscapes with rank $R=7$.  In Figure \ref{fig:Bro}, $100$ different random initial starters for ALS and RALS were used on tensor $\mathcal{X}$. The relative error ($\Vert \mathcal {X}^k -\mathcal{X}^{k-1}\Vert_F^2$) was used as the  stopping criterion, but the absolute error ($\Vert \mathcal{X}-\mathcal{X}_{\mbox{final}}\Vert_F^2$)  was measured as well. The table in Figure \ref{fig:Bro} shows that RALS performed slightly better than ALS with respect to both relative and absolute errors as well as the number of iterations.  
\begin{figure}[h!]
\begin{center}
\includegraphics[width=.6\textwidth]{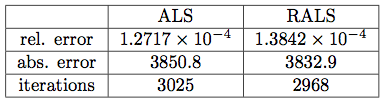}
\includegraphics[width=.70\textwidth]{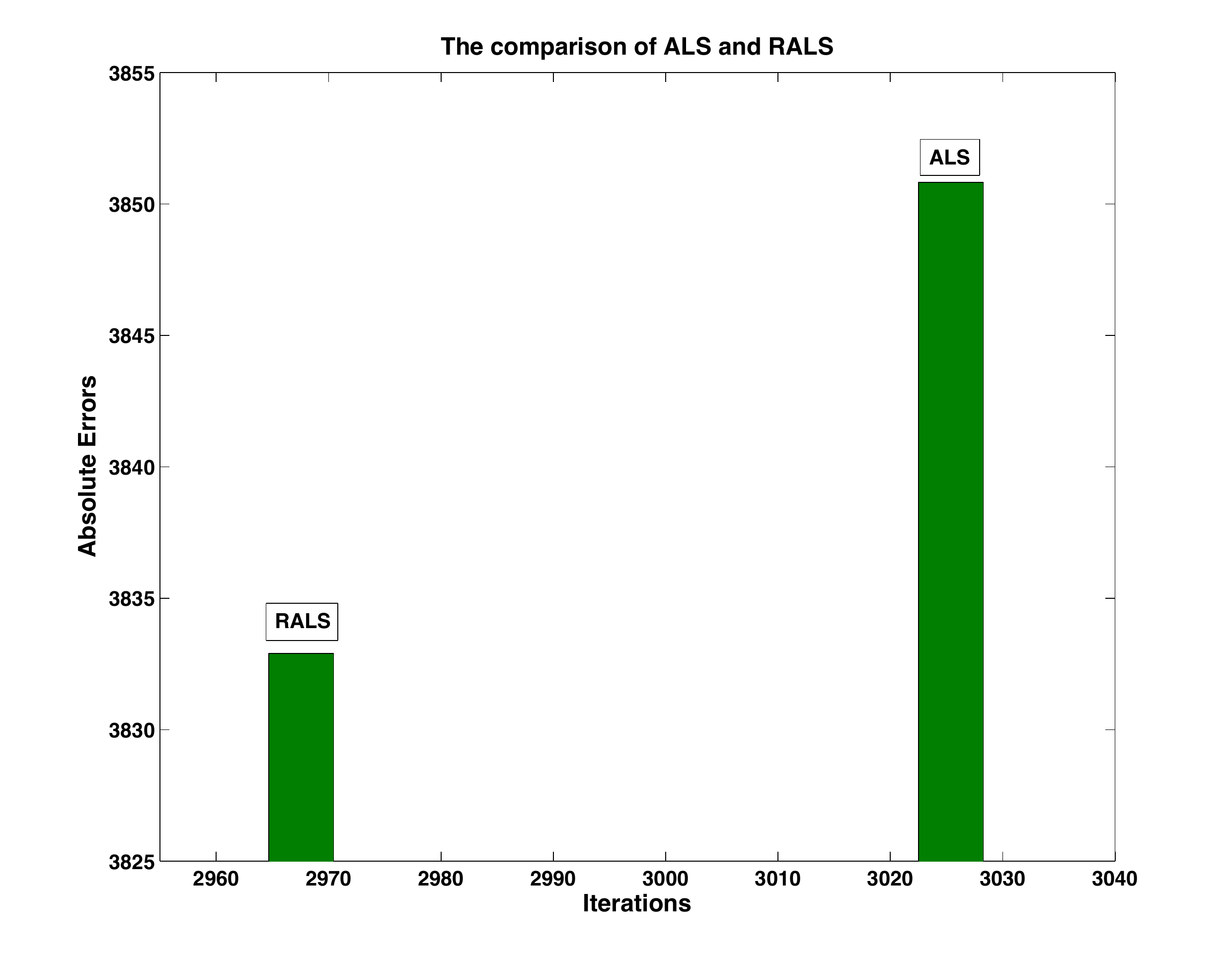}
\caption{A comparison of ALS and RALS for large data set \cite{Bro1}, averaging $100$ random initialized runs.}
\label{fig:Bro}
\end{center}
\end{figure}

\section{Conclusion}
The RALS  method proposed by Navasca, Kindermann and De Lathauwer \cite{NDLK} is a numerical technique for the classical problem of solving the CP decomposition of a given tensor.  We examined the RALS method to find some theoretical explanations on what we observed numerically.
In many instances, several examples showed that RALS converges faster than ALS. Moreover, RALS decreases the high number of ALS iterations, thereby removing the swamp to some degree. Furthermore, our numerical experiments provide us a numerical justification that ALS swamping is related to the rank deficiency of the Khatri-Rao products. This phenomena is not present when the RALS  algorithm is implemented. Based on these observations, it is important to study the theoretical properties of RALS and its differences from ALS. Both the ALS and the RALS are related to the GS and the proximal modification of GS (PGS), respectively, by vectorizing the three factor matrices in the cost functionals. Using the properties of PGS, we have proved that the limit point of a converging sequence obtained from the RALS algorithm is a critical point of the original ALS problem. Some difficulties arise when proving the same convergent results for ALS due to the lack of strict quasiconvexity.  These same difficulties are exhibited numerically as swamps.

\section*{Acknowledgements}

C.N. and N.L. are both in part supported by the U.S. National Science Foundation DMS-0915100. C.N. and N.L. are very grateful to Martin Mohlenkamp for useful conversations at the Tensor Workshop at the American Institute of Mathematics in Palo Alto. C.N. would like to thank the hospitality of the Industrial Mathematics Institute at Johannes Kepler Universit\"at Linz during her visits. We are also grateful to an anonymous  referee  for various helpful comments. 


\baselineskip=12pt

\begin{thebibliography}{10}

\bibitem{Acar}
E. Acar, C. A. Bingol, H. Bingol, R. Bro, and B. Yener.
\newblock Multiway analysis of epilepsy tensors.
\newblock {\em Bioinformatics,} {\bf 23} (13) (2007), i10-i18.

\bibitem{AAA} A. Auslender.
\newblock Asymptotic Properties of the Fenchel Dual Functional and Applications to Decomposition Problems.
\newblock {\it J. Optim. Theory Appl.} {\bf 73} (3) (1992), 427-499.

\bibitem{DPB} D.P. Bertsekas, {\em Nonlinear Programming}, Athena Scientific, Belmont, MA, 1995.

\bibitem{DPP} D.P. Bertsekas and P. Tseng.
\newblock Partial Proximal Minimization Algorithms for Convex Programming.
\newblock {\it  SIAM J. Optim.} {\bf 4} (3) (1994), 551-572.

\bibitem{Beylkin}
G. Beylkin and M. J. Mohlenkamp.
\newblock Numerical operator calculus in higher dimensions.
\newblock {\it Proceedings of the National Academy of Sciences}, {\bf
99}, (2002) 10246-10251.

\bibitem{BM2}
G.~Beylkin and M.J.~Mohlenkamp.
\newblock {\it Algorithms for numerical analysis in high dimensions.}
\newblock SIAM Journal on Scientific Computing, {\bf 26} (2005), 2133-2159.

\bibitem{JF} J.-F. Cardoso and P. Comon.
\newblock Independent Component Analysis, a survey of some algebraic methods.
\newblock {\it IEEE International Symposium on Circuits and Systems,} {\bf 2} (1996), 93-96.

\bibitem{JJ} J. Carrol and J. Chang
\newblock Analysis of Individual Differences in Multidimensional Scaling via an $N$-way Generalization of ``Eckart-Young" Decomposition.
\newblock {\it  Psychometrika} {\bf 9} (1970), 267-283.


\bibitem{dLdMvdW}
L. De Lathauwer, B. De Moor, and J. Vandewalle.
\newblock Computation of the canonical decompositionby means of a simultaneous generalized Schur decomposition.
\newblock {\it SIAM J. Matrix Anal. Appl.,} {\bf 26} (2004), 295-327.


\bibitem{JustLieven}
L. De Lathauwer.
\newblock A link between the canonical decomposition in multilinear algebra and simultaneous matrix diagonalization.
\newblock {\it SIAM J. Matrix Anal. Appl.,} {\bf 28} (2006), 642-666.

\bibitem{Lieven}
L. De Lathauwer.
\newblock {\em A Survey of Tensor Methods.}
\newblock ISCAS 2009, Taipei, Taiwan.

\bibitem{Lim}
V.~De Silva and L.-H.~Lim
\newblock Tensor rank and the ill-posedness of the best lowrank approximation problem. 
\newblock {\it SIAM J. Matrix Anal. Appl.}  {\bf 30} (3) (2008), 1084-1127.

\bibitem{DeVos}
M. De Vos, A. Vergult, L. De Lathauwer, W. De Clercq, S. Van Huffel, P.
Dupont, A. Palmini, and W. Van Paesschen, "Canonical decomposition
of ictal EEG reliably detects the seizure onset zone", {\em Neuroimage}, {\bf 37} (3) (2007), 844-854.


\bibitem{Doostan}
\newblock A.~Doostan, G.~Iaccarino, and N.~Etemadi.
\newblock {\it A least-squares approximation of high-dimensional
uncertain systems,}
\newblock Annual Research Briefs, Center for Turbulence Research,
Stanford University, 2007, 121-132.

\bibitem{HM} H. Engl, M. Hanke and A. Neubauer, {\em Regularization of Inverse Problems}, Kluwer Dordrecht, 1996.

\bibitem{Bro1}
N. K. M. Faber, R. Bro, and P. K. Hopke.
\newblock Recent developments in CANDECOMP/PARAFAC algorithms: A critical review.
\newblock {\it Chemometrics and Intelligent Laboratory Systems}, {\bf 65} (2003), 119Ð137.

\bibitem{LM} L. Grippo and M. Sciandrone.
\newblock On the convergence of the block nonlinear Gauss-Seidel method under convex constraints.
\newblock {\it Operations Research Letters,} {\bf 26} (2000), 127-136.

\bibitem{LGM} L. Grippo and M. Sciandrone.
\newblock Globally convergent block-coordinate techniques for unconstrained optimization.
\newblock {\it Optim. Methods Software,} {\bf 10} (1999), 587-637.

\bibitem{RAH} R. A. Harshman.
\newblock Foundations of the PARAFAC procedure: Model and Conditions for an ``Explanatory" Multi-code Factor Analysis.
\newblock {\it UCLA Working Papers in Phonetics,} {\bf 16} (1970), 1-84.

\bibitem{Hitch1}
F.L.~ Hitchcock. 
\newblock The expression of a tensor or a polyadic as a sum of products.
\newblock {\it Journal of Mathematics and Physics}, {\bf 6} (1927), 164-189.

\bibitem{Hitch2}
F.L.~ Hitchcock. 
\newblock Multilple invariants and generalized rank of a p-way matrix or tensor, 
\newblock {\it Journal of Mathematics and Physics}, {\bf 7} (1927), 39-79.

\bibitem{Jiang}
T.~Jiang and  N.D.~Sidiropoulos.
\newblock Kruskal's Permutation Lemma and the Identification of CANDECOMP/PARAFAC and Bilinear Models with Constant Modulus Constraints.
\newblock {\it IEEE Trans. on Signal Processing,} {\bf 52}, (2004), 2625-2636.

\bibitem{KB} T.G. Kolda and B.W. Bader.
\newblock Tensor Decompositions and Applications.
\newblock {SIAM Review}, {\bf 5} (3) (2009), 455-500.

\bibitem{Kroonenberg} 
P.M. Kroonenberg.
\newblock {\em Applied Multiway Data Analysis,}
\newblock Wiley, 2008.

\bibitem{JB} J.B. Kruskal.
\newblock Three-way arrays: Rank and uniqueness of trilinear decomposition with application to arithmetic complexity and statistics.
\newblock {\it Linear Algebra Appl.,} {\bf 18} (1977), 95-138.

\bibitem{JBK} J.B. Kruskal.
\newblock Rank, decomposition, and uniqueness for 3-way and N-way arrays,
\newblock {\it in Multiway Data Analysis,} R. Coppi and S. Bolasco, eds., North-Holldan, Amsterdam, pp. 7-18, 1989.



\bibitem{LimComon} 
L.-H. Lim and P. Comon.
\newblock Nonnegative Approximations of Nonnegative Tensors.
\newblock {\it J. Chemometrics,} {\bf 23} (2009), 432-441.

\bibitem{NDLK} 
C. Navasca, L. De Lathauwer and S. Kindermann.
\newblock {\em Swamp reducing technique for tensor decomposition,}
\newblock  in the 16th Proceedings of the European Signal Processing Conference, Lausanne, August 2008.

\bibitem{Nion}
D. Nion and L. De Lathauwer.
\newblock An enhanced line search scheme for complex-valued tensor decompositions. Application in DS-CDMA.
\newblock {\it Signal Processing,} {\bf 88} (2008), 749-755.

\bibitem{PP}
P.~Paatero.
\newblock Construction and analysis of degenerate PARAFAC models.
\newblock {\it  J. Chemometrics},  {\bf 14} (2000), 285-299.  

\bibitem{PP2}
P.~Paatero.
\newblock A weighted non-negative least squares algorithm for three-way PARAFAC factor analysis.
\newblock {\it Chemometrics Intell. Lab. Syst.,} {\bf 38} (1997), 223-242.

\bibitem{MJD} 
M.J.D. Powell.
\newblock On search directions for minimization algorithms.
\newblock {\it Math. Programming,} {\bf 4} (1973), 193-201.

\bibitem{Comon}
M. Rajih and P. Comon.
\newblock {\it Enhanced line search: A novel method to accelerate PARAFAC,}
\newblock in the Proceedings of the 13th European Signal Processing Conference, August 2005.



\bibitem{Sid1} 
N.D. Sidiropoulos, G.B. Giannakis, and R. Bro.
\newblock  Blind PARAFAC receivers for DS-CDMA systems.
\newblock {\em IEEE Trans. on Signal Processing,} {\bf 48} (3) (2000), 810-823.

\bibitem{Sid2} 
N. Sidiropoulos, R. Bro, and G. Giannakis.
\newblock Parallel factor analysis in sensor array processing.
\newblock {\em IEEE Trans. Signal Processing,} {\bf 48} (2000), 2377-2388.

\bibitem{Smilde}
A. Smilde, R. Bro, and P. Geladi.
\newblock {\em Multi-way Analysis. Applications in the Chemical Sciences.} 
\newblock Chichester, U.K., John Wiley and Sons, 2004.



\bibitem{GR} G. Tomasi and R. Bro.
\newblock {A comparison of algorithms for fitting the PARAFAC model}
\newblock {\it Computational Statistics \& Data Analysis,} {\bf 50} (2006), 1700-1734.

\bibitem{Bro1}
J.~P.~Wold, R.~Bro, A.~Veberg, F.~Lundby, A.~N.~Nilsen and J.~Moan.
\newblock Active Photosensitizers in Butter Detected by Fluorescence Spectroscopy and Multivariate Curve Resolution
\newblock {\em J.Agric. Food Chem}. {\bf 54} (2006), 10197-10204.


\end{thebibliography}
\end{document}